\PassOptionsToPackage{table,xcdraw,dvipsnames}{xcolor}
\documentclass[review,onefignum,onetabnum]{siamart171218}

\usepackage{lipsum}
\usepackage{amsfonts}
\usepackage{graphicx}
\usepackage{epstopdf}
\usepackage{algorithmic}
\ifpdf
  \DeclareGraphicsExtensions{.eps,.pdf,.png,.jpg}
\else
  \DeclareGraphicsExtensions{.eps}
\fi

\usepackage{cite}
\usepackage{amsmath,amssymb,amsfonts}
\usepackage{blkarray}
\usepackage{algorithmic}
\usepackage{textcomp}
\usepackage{subcaption}
\usepackage{bm}
\usepackage{mathtools}

\newcommand{\N}{\mathbb{N}}
\newcommand{\Z}{\mathbb{Z}}

\newcommand{\C}{C_{\mathrm{M}}}
\newcommand{\D}{D_{\mathrm{M}}}
\newcommand{\setC}{\mathcal{C}}
\newcommand{\setD}{\mathcal{D}}
\newcommand{\X}{\mathcal{X}}
\newcommand{\setI}{\mathcal{I}}
\newcommand{\setO}{\mathcal{O}}
\newcommand{\setS}{\mathcal{S}}
\newcommand{\setJ}{\mathcal{J}}

\newcommand{\norm}[1]{\left\lVert#1\right\rVert}
\newcommand{\ceil}[1]{\left\lceil#1\right\rceil}
\newcommand{\floor}[1]{\left\lfloor#1\right\rfloor}
\newcommand{\x}{\bm{x}}
\newcommand{\y}{\bm{y}}
\newcommand{\z}{\bm{z}}
\newcommand{\tb}{\textcolor{blue}}
\newcommand{\tr}{\textcolor{red}}
\newcommand{\tg}{\textcolor{ForestGreen}}

\DeclareMathOperator*{\argmin}{arg\,min}

\newsiamremark{remark}{Remark}
\newsiamremark{example}{Example}
\newsiamremark{hypothesis}{Hypothesis}
\crefname{hypothesis}{Hypothesis}{Hypotheses}
\newsiamthm{claim}{Claim}
\newsiamthm{assumption}{Assumption}

\headers{Coordinating, Anticoordinating, and Imitating}{H. Le, M. Rajaee, and P. Ramazi}

\title{Heterogeneous Mixed Populations of Coordinating, Anticoordinating, and Imitating Individuals\thanks{
Submitted to the editors DATE.
\funding{This project was partly funded by Alberta Environment and Parks. 
We are also grateful for the generous funding provided by Prof. Russell Greiner.}}}

\author{Hien Le\thanks{Faculty of Science, University of Alberta, Canada  
  (\email{hienthut@ualberta.ca}).}
  \and Mohaddeseh Rajaee\thanks{Department of Electrical and Computer Engineering, Isfahan University of Technology, Iran(\email{mohaddesehrajaee@gmail.com}).}
\and Pouria Ramazi\thanks{Department of Mathematics and Statistics, Brock University, Canada 
  (\email{p.ramazi@gmail.com}, \email{pramazi@brocku.ca}).}}

\usepackage{amsopn}

\makeatletter
\newcommand*{\addFileDependency}[1]{
  \typeout{(#1)}
  \@addtofilelist{#1}
  \IfFileExists{#1}{}{\typeout{No file #1.}}
}
\makeatother

\newcommand*{\myexternaldocument}[1]{%
    \externaldocument{#1}%
    \addFileDependency{#1.tex}%
    \addFileDependency{#1.aux}%
}

\ifpdf
\hypersetup{
  pdftitle={An Example Article},
  pdfauthor={D. Doe, P. T. Frank, and J. E. Smith}
}
\fi

\myexternaldocument{ex_supplement}

\begin{document}
\nolinenumbers
\maketitle

\begin{abstract}
  Decision-making individuals are typically either an \emph{imitator}, who mimics the action of the most successful individual(s), 
    a \emph{conformist} (or \emph{coordinating individual}), who chooses an action if enough others have done so, or a \emph{nonconformist} (or \emph{anticoordinating individual}), who chooses an action if few others have done so. 
    Researchers have studied the asymptotic behavior of populations comprising one or two of these types of decision-makers, but not altogether, which we do for the first time. 
    We consider a population of heterogeneous individuals, each either \emph{cooperates} or \emph{defects}, and earns payoffs according to their possibly unique payoff matrix and the total number of cooperators in the population.
    Over a discrete sequence of time, the individuals revise their choices asynchronously based on the \emph{best-response} or \emph{imitation} update rule.
    Those who update based on the best-response are a conformist (resp. nonconformist) if their payoff matrix is that of a coordination (resp. anticoordination) game.
    We take the distribution of cooperators over the three types of individuals with the same payoff matrix as the state of the system.
    First, we provide our simulation results, showing that a population may admit zero, one or more equilibria at the same time, and several non-singleton minimal positively invariant sets. 
    Second, we find the necessary and sufficient condition for equilibrium existence.
    Third, we perform stability analysis and find that only those equilibria where the imitators either all cooperate or all defect are likely to be stable. 
    Fourth, we proceed to the challenging problem of characterizing the minimal positively invariant sets and find conditions for the existence of such sets. Finally, we study the stochastic stability of the states under the perturbed dynamics, where the agents are allowed to make mistakes in their decisions with a certain small probability.  
\end{abstract}



\section{Introduction}
Decision-making individuals typically take action based on either the success or frequency of their fellows' choices \cite{van2015focus}.
The first group simply imitates the choice of successful others, referred to as \emph{imitators}. 
The second, decide based on how many others have chosen a particular action:
some choose an action if enough others have already taken that action, referred to as \emph{conformists} or \emph{coordinating individuals}, whereas some choose an action if few others have done so, referred to as \emph{nonconformists} or \emph{anticoordinating individuals}.
Whether others choosing a particular action are ``enough'' or ``few'' for the individual to choose that action is determined by her \emph{threshold} or the so called \emph{temper}.
All three types are common in both human societies and nature \cite{cunningham2002empirical,smith1974theory,zhou2006innovation}.
In technology markets, conformists invest in common products to avoid risk, nonconformists develop rare products to benefit from a monopoly \cite{collins2015}, and imitators follow highest-earning firms \cite{bursztyn2014understanding}.
When deciding on whether to vaccine their infants, some mothers (conformists) promote vaccination as they do not want to diverge from the cultural norm; others (nonconformists) inhibit vaccination as they feel that their child is not at risk since other children are vaccinated; and some (imitators) may be in touch with an influential homeopath/naturopath \cite{allen2010parental}.
The fundamental questions regarding populations consisting of these types of decision-makers are concerned with their asymptotic collective behavior.
More specifically, is it possible for the population to eventually settle and reach an equilibrium where all individuals are satisfied with their decisions? If yes, how stable is the equilibrium to changes in the decisions?
More importantly, if the population never settles, and undergoes perpetual fluctuations, can we characterize the fluctuations, e.g., find their lengths?

\emph{Evolutionary game theory} models a decision-making process by a population of interactive agents who choose between typically two strategies \emph{cooperation} and \emph{defection}, accordingly earn payoffs based on their payoff matrices, and revise their strategies based on some update rule \cite{sandholm2010population}.
Imitators update based on the \emph{imitation update rule} and choose the strategy of the individual(s) with the highest payoff, and the other two types update based on the \emph{best-response update rule} and choose the strategy that maximizes their payoffs.
Based on their payoff matrices, best-responders are either a conformist or nonconformist, and are modeled by \emph{threshold models}, where a conformist chooses cooperation if the total number of cooperators in the population exceeds her threshold and \emph{vice versa} \cite{riehl2018survey}.
Populations of homogeneous individuals \cite{ramazi2014stability}, who have the same payoff matrix or thresholds, are analyzed under both best-response  \cite{lelarge2012diffusion,young2011dynamics,oyama2015sampling,alos2003finite, kreindler2013fast} and imitation update rules \cite{cimini2015dynamics, govaert2017convergence}. 
Heterogeneous population of conformists, and heterogeneous population of nonconformists are both known to equilibrate and the same holds for imitators with the \emph{coordination} payoff matrix \cite{ramazi2016networks,riehl2018survey, ramazi2018asynchronous}.
However, imitators earning according to the \emph{anticoordination} payoff matrix are reported and proven to oscillate in their strategies under a variety of conditions \cite{imitation,henderson2016alternative}.
Recently, mixed populations of nonconformists and imitators who earn based on the anticoordination payoff matrix are shown to equilibrate if and only if they admit an equilibrium \cite{le2020heterogeneous}.
We have performed stability analysis for population dynamics in \cite{ramazi2018asynchronous} and later \cite{le2020heterogeneous}.
No study has reported the asymptotic behavior of mixed-populations of conformists and nonconformists. 
More generally, it remains open whether mixed-populations of all three types of decision-makers, and under both the coordination and anticoordination payoff matrices, admit a stable or unstable equilibrium.
Moreover, none of the above studies have addressed the challenging problem of characterizing the oscillatory behavior of the population dynamics. 

Another branch of research in evolutionary game theory considers agents that are prone to tremble randomly in decision making. 
This perturbation changes the chance of different outcomes to be observed in the long run. 
Those states that are robust to the perturbation in the dynamics when the trembling rate vanishes are called \emph{stochastically stable} states. T
he notion of stochastic stability, first studied in~\cite{kandori1993learning,foster1990stochastic,young1993evolution}, provides also a refinement tool for dynamics with multiple minimal invariant sets. 
The stochastic stability analysis for a certain type of homogeneous population of conformists is performed in~\cite{kandori1993learning}, where authors provide conditions under which the only stochastically stable state is the risk dominant equilibrium. 
Ellison generalizes this result to homogeneous populations with arbitrary number of strategies for agents to play~\cite{ellison2000basins}. 
The stochastic stability of mixed populations of imitators and best-responders remains concealed.

We consider a mixed-population of conformists, nonconformists, and imitators who decide between cooperation and defection and revise their choices over a discrete sequence of time. 
We follow the evolutionary game theory framework where individuals earn according to their payoff matrices, which can be unique to each individual, resulting in a heterogeneous population, and can be both of the coordination and anticoordination game. 
First, we provide numerical examples and show that the resulting population dynamics can admit zero or more equilibria as well as a set where the dynamics perpetually fluctuate between several states, which turns out to be a non-singleton minimal positively invariant set.
Then we find the necessary and sufficient condition for equilibrium existence and identify all possible equilibria that the dynamics may possess.
We interpret the results by introducing the notion of \emph{cooperation-preserving} groups who tend to cooperate if they are the only cooperators in the population.
Next, we adjust the notion of equilibrium stability for general discrete population dynamics and find the necessary and sufficient condition for an equilibrium to be stable.
We find that only those extreme equilibria where all imitators defect or all cooperate are likely to be stable. 
Next, we proceed to characterize the minimal positively invariant sets.
We approximate the sets by finding necessary conditions for a general set to be positively invariant. 
We also introduce a set that can be invariant, which helps in limiting our search when looking for a minimal positively invariant set.  
Finally, we study the situation when the agents may tremble and make errors in their decisions with a certain small probability, resulting in perturbed dynamics. We investigate the stochastically stable states and find that, under specific conditions, for mixed binary-type populations, if there exist stochastically stable equilibria, an extreme equilibrium has to be one.

Our results reveal the possible outcomes of one of the most general population dynamics considered in the literature. According to the stability analysis, we may not expect a population to both settle and preserve the diversity of strategies among the imitators, unless small deviations in the imitators' strategies at that settling state result in another settling state. The stochastic stability analysis on the other hand, imply that if an equilibrium consisting of both imitating cooperators and defectors is persistently visited under agents' trembles, so does an extreme equilibrium. The results, therefore, highlights the``robustness" of extreme equilibria under perturbations.  
Moreover, the analysis on the invariant sets allows us to estimate their cardinality and other properties such as the minimum and the maximum number of cooperators in the set. Our results open the door to future convergence analysis of the population dynamics. 

\section{Model}\label{sec:model}
We consider a well-mixed population of $n$ agents playing 2-player games over a discrete time sequence.
At each time $t= 0, 1,\ldots$, every agent $i\in \{1, \ldots, n\}$ plays either \emph{cooperation} ($C$) or \emph{defection} ($D$) and accordingly earns a payoff against each of her opponents based on the payoff matrix
\begin{equation} \label{pmatrix}
\pi_{i}=
\begin{blockarray}{ccc}
& C & D\\
\begin{block}{c(cc)}
  C & R_{i} & S_{i}\\
  D & T_{i} & P_{i}, \\
\end{block}
\end{blockarray}, \qquad (R_i+ P_i) \neq (T_i+ S_i),
\end{equation}
where $R_{i}, S_{i}, T_{i}$ ad $P_{i}$ are the payoffs of $C$-against-$C$, $C$-against-$D$, $D$-against-$C$ and $D$-against-$D$. 
Denote the strategy of agent $i$ by $s_i\in\{C, D\}$
and the total number of cooperators in the population by $n^C$.
Then the total payoff (utility) of agent $i$ against the population is
\begin{equation*}
    \begin{cases}
        u_{i}^{C} = n^{C}(R_{i}-S_{i})+nS_{i} &\text{if } s_i = C \\
        u_{i}^{D} = n^{C}(T_{i}-P_{i})+nP_{i} &\text{if }
        s_i = D
    \end{cases},
\end{equation*}
where $u^C_i$ and $u^D_i$ are the utility of agent $i$ upon playing cooperation and defection.
At each time $t$, an agent becomes active to revise her strategy at time $t+1$ based on either the \emph{best-response} or \emph{imitation update rule}.
Each agent follows only one of these two over time, and is correspondingly called a \emph{best-responder} or an \emph{imitator}.
If agent $i$ is a best-responder, she chooses the strategy that maximizes her utility; should cooperation and defection provide her the same utility, she would stick to her current strategy.
Namely, upon activation at time $t$, agent $i$ updates her strategy as
\begin{equation*} \label{general_br_update_rule}
    s_{i}^B\left(t+1\right)=
    \begin{cases}
        C
        &\text{if } u_{i}^{C}\left(t\right) > u_{i}^{D}\left(t\right)\\
        D
        &\text{if } u_{i}^{D}\left(t\right) > u_{i}^{C}\left(t\right)\\
        s_{i}^B\left(t\right) &\text{if } u_{i}^{D}\left(t\right) = u_{i}^{C}\left(t\right)
\end{cases},
\end{equation*}
where the superscript $B$ in $s_i^B$ indicates that agent $i$ is a best-responder.
By defining the \emph{temper} of agent $i$ as $\tau_i \triangleq n(P_i-S_i)/(R_i+P_i-S_i-T_i)$, the above equation can be simplified as follows when $R_i+P_i-S_i-T_i > 0$:
\begin{equation} \label{abr_update_rule}
    s_{i}^B\left(t+1\right)=
    \begin{cases}
        C
        &\text{if } n^C(t) > \tau_i\\
        D
        &\text{if } n^C(t) < \tau_i\\
        s_{i}^B\left(t\right) &\text{if } n^C(t) = \tau_i
\end{cases}.
\end{equation}
We call an agent following this update rule a {conformist} or coordinating agent as she coordinates her strategy with `those considerable in number'; that is, she cooperates (resp. defects) if the number of cooperators (resp. defectors) is greater (resp. less) than her temper.
Similarly, if $R_i+P_i-S_i-T_i < 0$, the update rule is simplified to
\begin{equation} \label{cbr_update_rule}
    s_{i}^B\left(t+1\right)=
    \begin{cases}
        C
        &\text{if } n^C(t) < \tau_i\\
        D
        &\text{if } n^C(t) > \tau_i\\
        s_{i}^B\left(t\right) &\text{if } n^C(t) = \tau_i
\end{cases}.
\end{equation}
We refer to an agent following this rule a {nonconformist} or anticoordinating agent as she chooses the opposite of the strategy of `those considerable in number'.
So a best-responder is either a conformist or an nonconformist\footnote{
If the equality in \eqref{pmatrix} takes place, then either $u_i^D(t) = u_i^C(t)$ for all $t$, or one of the strategies becomes dominant and is always favored. 
In either case, agent $i$ fixes her strategy the first time she updates and can be modelled by both a conformist and an nonconformist with $\tau\in\{-0.5,n+0.5\}$.}.
An imitator, on the other hand, ignores her own utility and copies the strategy of the agent earning the highest utility; if there are two highest-earners with two different strategies, she will not switch strategies.
So upon activation at time $t$, agent $i$ who is an imitator updates her strategy as
\begin{equation} \label{general_im_update_rule}
    s_{i}^I(t+1)=
    \begin{cases}
        C &\text{if } \C(t) > \D(t) \\
        D & \text{if } \C(t) < \D(t) \\
        s_{i}^I(t) &\text{if } \C(t) = \D(t)
    \end{cases}, 
\end{equation}
where the superscript $I$ in $s_i^I$ indicates that agent $i$ is an imitator, and $\C(t)$ and $\D(t)$ are the utility of the highest-earning cooperator and defector at time $t$:
\begin{gather*}
    \C(t) = \sup_{i \in \setC(t)} u_i^C(t)\quad \setC(t) = \{j \in \{1, \ldots, n\}\,|\,s_j(t) = C\}, \\
    \D(t) = \sup_{i \in \setD(t)} u_i^D(t) \quad\!\! \setD(t) = \{j \in \{1, \ldots, n\}\,|\, s_j(t) = D\},
\end{gather*}
where given a set $\setS$, $\sup \mathcal{S}$ is the \emph{supremum of $\mathcal{S}$}.
By defining, $\sup \setS \triangleq -\infty$ when $\setS = \emptyset$, we correctly include the case $\setC(t) = \emptyset$ (resp. $\setD(t)= \emptyset$) when all agents including the highest earner are defecting (resp. cooperating) since the active imitator then tends to defect (resp. cooperate): $\D(t) > \C(t) = -\infty$ (resp. $\C(t) > \D(t) = -\infty$).
Each of the above three update rules are commonly used in the literature, either in their current or similar forms \cite{tsakas2014imitating,li2019control,swenson2017robustness,bravo2015reinforcement,chapman2013convergent, wang2020evolution}, under stochastic setups \cite{tan2016emerging, montanari2010spread, sandholm2010population}, under different games \cite{etesami2019simple, govaert2017convergence}, or when the population dynamics are approximated by continuous mean field dynamics \cite{laraki2015inertial, diekmann2009cyclic}.  


The agents can be thought of as investors in the stock market who either \emph{buy} or \emph{sell}.
Some simply follow the strategy of the most successful investors.
Some choose to buy when they observe enough others doing so, because they believe that the market price will then increase.
Others choose to sell when enough others are buying since they can then sell at high prices.
These three types of investors are common in the stock market, and their coexistence is evident.
The three types of individuals also coexist when people decide whether or not to get vaccinated \cite{bodine2013conforming}.
When the vaccination coverage is high, some believe that they are relatively safe and decide not to have the vaccination.
Meanwhile, under the same situation, some choose to get vaccinated because they worry that the danger of the disease is the reason people get vaccinated.
Others do not put much thought into the problem and just follow the most healthy individuals.

While all imitators behave in the exact same way, best-responders with different payoff matrices may choose different strategies upon activation.
We, thus, group the best-responders with the same payoff matrix into the same \emph{type} and assume that there are all together $b$ types of nonconformists, which we refer to as the \emph{anticoordinating types}, and $b'$ types of conformists, which we refer to as the \emph{coordinating types}.
Now note that best-responders of the same type share the same temper, and unlike conformists, nonconformists' tendency to cooperate increases with their tempers.
Therefore, we label the anticoordinating types in the descending order of their tempers by $1,2, \ldots, b$ and the coordinating types in the ascending order of their tempers by $1, 2, \ldots, b'$.
By denoting the temper of a type-$i$ nonconformist as $\tau_i^a$ and the temper of a type-$i$ conformist as $\tau_i^c$ and assuming distinct tempers, we obtain
\begin{equation}\label{TemperOrder}
    \tau_1^a>\tau_2^a>\ldots>\tau_b^a 
\quad
\text{and}
\quad
    \tau_1^c < \tau_2^c < \ldots < \tau_{b'}^c.
\end{equation}
We consider only the case where the tempers are not integers to avoid unnecessary complications (see \cite{ramazi2020convergence} for a framework for integer tempers).
We take the distribution of cooperators over the imitators and the two types of the best-responders as the \emph{population state}:
\begin{equation*}
    \x= (x^I, x_1^a, \ldots, x_b^a, x_{b'}^c, \ldots, x_1^c),
\end{equation*}
where $x^I$ is the number of cooperating imitators, and $x_i^a$ and $x_i^c$ denote the number of cooperating nonconformists and cooperating conformists of type-$i$.
Hence, the state space is
\begin{align*}
    \X = \Big\{(x^I, x_1^a, \ldots, x_b^a, x_{b'}^c, \ldots, x_1^c)\,|\,
    & x^I \in \{0, \ldots, m\}, \\
    & x_i^a \in \left\{0, \ldots, n_i^a\right\} \, \forall i \in \{1, \ldots, b\},\\
    & x_i^c \in \left\{0,\ldots, n_i^c\right\} \, \forall i \in \{1, \ldots, b'\}
    \Big\},
\end{align*}
where $m >0$ is the number of imitators, and $n_i^a$ and $n_i^c$ are the number of type-$i$ nonconformists and conformists. 
We assume that associated with every payoff matrix, there is at least one best-responder who earns accordingly; namely, the payoff matrices of the imitators are included in those of the best-responders.
Therefore, each imitator falls into one of the coordinating and anticoordinating types, although the imitators are neither conformists nor nonconformists.
Since agents of the same type earn the same payoff if they play the same strategy, we denote the utilities of type-$i$ anticoordinating cooperators and defectors by $C_i^a$ and $D_i^a$, and the utilities of type-$i$ coordinating cooperators and defectors by $C_i^c$ and $D_i^c$.

The agents become active according to an \emph{activation sequence}, which is an infinite sequence of agents $(i^t)_{t=0}^{\infty}$, where $i^t$ is the active agent at time $t$.
The activation sequence is \emph{asynchronous}; namely, at every time $t\geq 0$, exactly one agent becomes active to update her strategy at time $t+1$.
We do not make any assumption on the activation sequence, allowing it to be randomly generated or depend on past played actions.
Update rules \eqref{abr_update_rule}, \eqref{cbr_update_rule} and \eqref{general_im_update_rule} together with the activation sequence of the agents govern the dynamics of $\x(t)$, which we refer to as the \emph{population dynamics}.
Indeed, the dynamics can be seen as a \emph{multivalued dynamical system} as the sate $x(t)$ may end up at different states at the next time step, based on the activation sequence.
We are interested in the limit sets of the dynamics.
An \emph{equilibrium} of the population dynamics is a state where every agent is satisfied with her strategy. 
More specifically, we define an \emph{equilibrium} to be a population state $\x^*\in\X$ such that if the solution trajectory starts from that state, it remains there afterwards under any activation sequence, i.e., $\x(0)=\x^* \Rightarrow \x(t) = \x^*$ for all $t \geq 0$ and any $(i^t)_{t=0}^{\infty}$.
According to \cite{le2020heterogeneous}, mixed populations of best-responders and imitators who play anticoordination games, will reach an equilibrium if and only if the dynamics admit an equilibrium.
However, when both coordination and anticoordination games are played in the population, which is the case with our setup, the population may never settle down and exhibit perpetual fluctuations \cite{ramazi2016networks}, even if the dynamics admit an equilibrium.
We simulate some of the asymptotic outcomes of the dynamics in the following section.
The examples also provide useful hints in developing the results in the later sections.
\section{Numerical examples} \label{sec_examples}
\begin{example}
    Consider a population of $75$ agents with two anticoordinating and three coordinating types.
    The distribution of the population over the imitators and different types of best-responders is
    $(\tr{m}, \tg{n_1^a}, \tg{n_2^a}, \tb{n_3^c}, \tb{n_2^c}, \tb{n_1^c}) = (\tr{20}, \tg{9}, \tg{20}, \tb{10}, \tb{1}, \tb{15})$, 
    The payoff matrices are set so that they result in the following utility functions:
    \begin{align*}
        &C_1^a(n^C) = -\frac{16}{13} n^C+\frac{623}{13}, &&D_1^a(n^C) = \frac{36}{13} n^C-\frac{768}{13};\\
        &C_2^a(n^C) = -\frac{80}{79} n^C+\frac{4375}{79}, &&D_2^a(n^C) = 45;\\
        &C_3^c(n^C) = \frac{72}{43} n^C-\frac{1945}{43}, &&D_3^c(n^C) = -\frac{76}{43} n^C +\frac{4086}{43};\\
        &C_2^c(n^C) = \frac{4}{3} n^C-23, &&D_2^c(n^C) = -\frac{2}{3}n^C+40;\\
        &C_1^c(n^C) = \frac{48}{13} n^C - \frac{855}{13}, &&D_1^c(n^C) = \frac{24}{13} n^C-\frac{291}{13}.
    \end{align*}
    Therefore, the tempers are
    \begin{equation*}
        (\tg{\tau_1^a}, \tg{\tau_2^a}, \tb{\tau_3^c}, \tb{\tau_2^c}, \tb{\tau_1^c}) = (\tg{26.8}, \tg{10.3}, \tb{40.8}, \tb{31.5}, \tb{23.5}).
    \end{equation*}
    The population dynamics admit multiple equilibria: $(\tr{0}, \tg{9}, \tg{0},\tb{0}, \tb{0}, \tb{15})$, $(\tr{20}, \tg{0}, \tg{0}, \tb{0}, \tb{1}, \tb{15})$, $(\tr{20}, \tg{0}, \tg{0}, \tb{10}, \tb{1}, \tb{15})$, and $(\tr{15}, \tg{0}, \tg{0}, \tb{0}, \tb{0}, \tb{15})$.
    We observe that at any equilibrium, there are a benchmark type $j$ of nonconformists and a benchmark type $j'$ of conformists such that all nonconformists of types $1, \ldots, j$ (resp. $j+1, \ldots, b$) and all conformists of types $1, \ldots, j'$ (resp. $j'+1, \ldots, b'$) cooperate (resp. defect).
    At $(\tr{0}, \tg{9}, \tg{0},\tb{0}, \tb{0}, \tb{15})$, for example, the nonconformists' benchmark type is $\tg{1}$ and the conformists' benchmark type is $\tb{1}$.
    The total number of cooperators is $24$, falling short of the temper of type-$1$ nonconformists and exceeding the temper of  type-$1$ conformists.
    Hence, these agents continue cooperating upon activation.
    nonconformists of type $2$ (resp. conformists of types $2$ and $3$) also do not switch strategies, because they are defecting and the number of cooperators exceeds (resp. falls short of) their temper.
    Moreover, none of the imitators switch strategies since they are defecting and coordinating defectors of type $3$ are the highest earners.
\end{example}
\begin{example}\label{fluc_equil_eg}
    Consider a population of $75$ agents belonging to two anticoordinating and three coordinating types.
    The population is distributed among the imitators and best-responders as
    $(\tr{m}, \tg{n_1^a}, \tg{n_2^a}, \tb{n_3^c}, \tb{n_2^c}, \tb{n_1^c}) = (\tr{14}, \tg{9}, \tg{20}, \tb{10}, \tb{1}, \tb{15})$.
    Among the $14$ imitators, $4$ of them have the same payoff matrix as that of type-$1$ nonconformists.
    The payoffs are set so that they result in the following utility functions:
    \begin{align*}
        &C_1^a(n^C) = -\frac{40}{13} n^C+\frac{1655}{13}, &&D_1^a(n^C) = \frac{36}{13} n^C-\frac{378}{13};\\
        &C_2^a(n^C) = \frac{16}{79} n^C+\frac{1495}{79}, &&D_2^a(n^C) = \frac{64}{79} n^C+\frac{1003}{79};\\
        &C_3^c(n^C) = \frac{24}{43} n^C-\frac{505}{43}, &&D_3^c(n^C) = -\frac{40}{43} n^C +\frac{2103}{43};\\
        &C_2^c(n^C) = \frac{4}{3} n^C-23, &&D_2^c(n^C) = -\frac{2}{3}n^C+40;\\
        &C_1^c(n^C) = \frac{8}{13} n^C + \frac{345}{13}, &&D_1^c(n^C) = -\frac{16}{13} n^C+\frac{909}{13}.
    \end{align*}
    The tempers are, hence, 
    \begin{equation*}
        (\tg{\tau_1^a}, \tg{\tau_2^a}, \tb{\tau_3^c}, \tb{\tau_2^c}, \tb{\tau_1^c}) = (\tg{26.8}, \tg{10.3}, \tb{40.8}, \tb{31.5}, \tb{23.5}).
    \end{equation*}
    The population dynamics exhibit different asymptotic outcomes for the same initial condition where all agents defect (\cref{fig:example}). 
    Depending on the activation sequence, the population may reach an equilibrium (\cref{fig:xd}) or undergo perpetual fluctuations (\cref{fig:xf}).
    In the second case, all type-$2$ nonconformists and all conformists fix their strategies after some finite time $T$.
    This is guaranteed since the number of cooperators fluctuates between $26$ and $27$ after time $T$.
    When the number of cooperators is $26$, the agents who have not fixed their strategies, i.e., the imitators and type-$1$ nonconformists, do not switch from cooperation to defection, because $26$ falls short of the nonconformists' tempers and a cooperator of anticoordinating type $1$ is earning the highest payoff.
    Such a cooperator exists, because if all type-$1$ anticoordinating agents, i.e., the 9 nonconformists and 4 of the imitators, were defecting, 
    then the number of cooperators would be at most $25$.
    Similarly, one can verify that the number of cooperators never exceeds $27$.
    \begin{figure*}[h!]
            \begin{subfigure}[t]{0.49\textwidth}
            \centering
            \includegraphics[trim ={1.5cm 5cm 0 5cm}, clip, width=\textwidth]{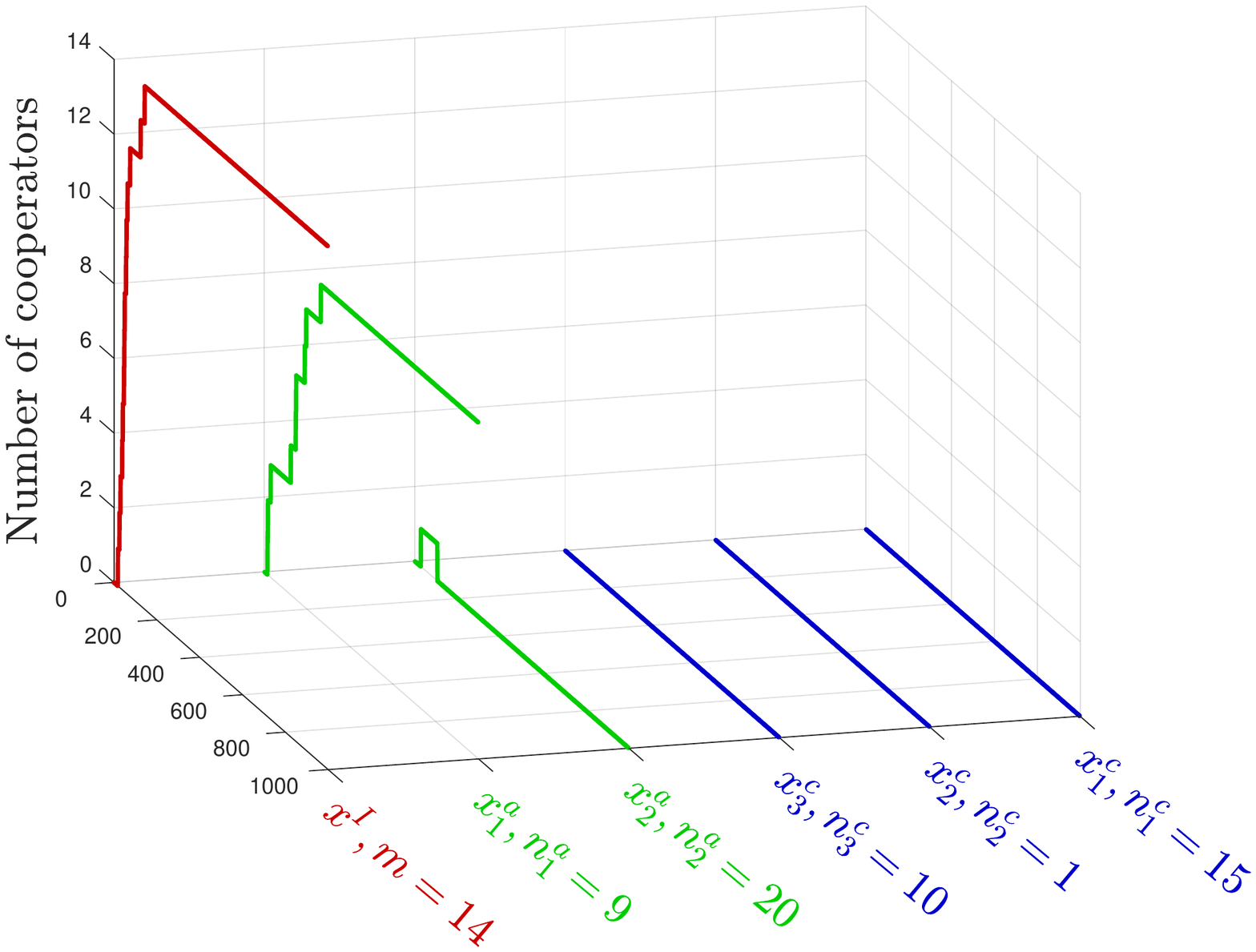}
            \caption{
            \textbf{Reaching a mixed equilibrium.}
            The population reaches the equilibrium $(14,9,0,0,0,0)$.
            The total number of cooperators at this state is $23$, which exceeds the temper of type-$2$ nonconformists and falls short of the tempers of the conformists. 
            Hence no defecting best-responder switch strategies.
            Moreover, $23$ is less than the temper of type-$1$ nonconformists, so they keep cooperating. 
            The imitators do not change their strategies either as the highest earners at this state are type-$1$ nonconformists, who are cooperating.}
            \label{fig:xd} 
        \end{subfigure}
        \hfill
        \begin{subfigure}[t]{0.49\textwidth}
            \centering
            \includegraphics[trim ={1.5cm 5cm 0 5cm}, clip, width=\textwidth]{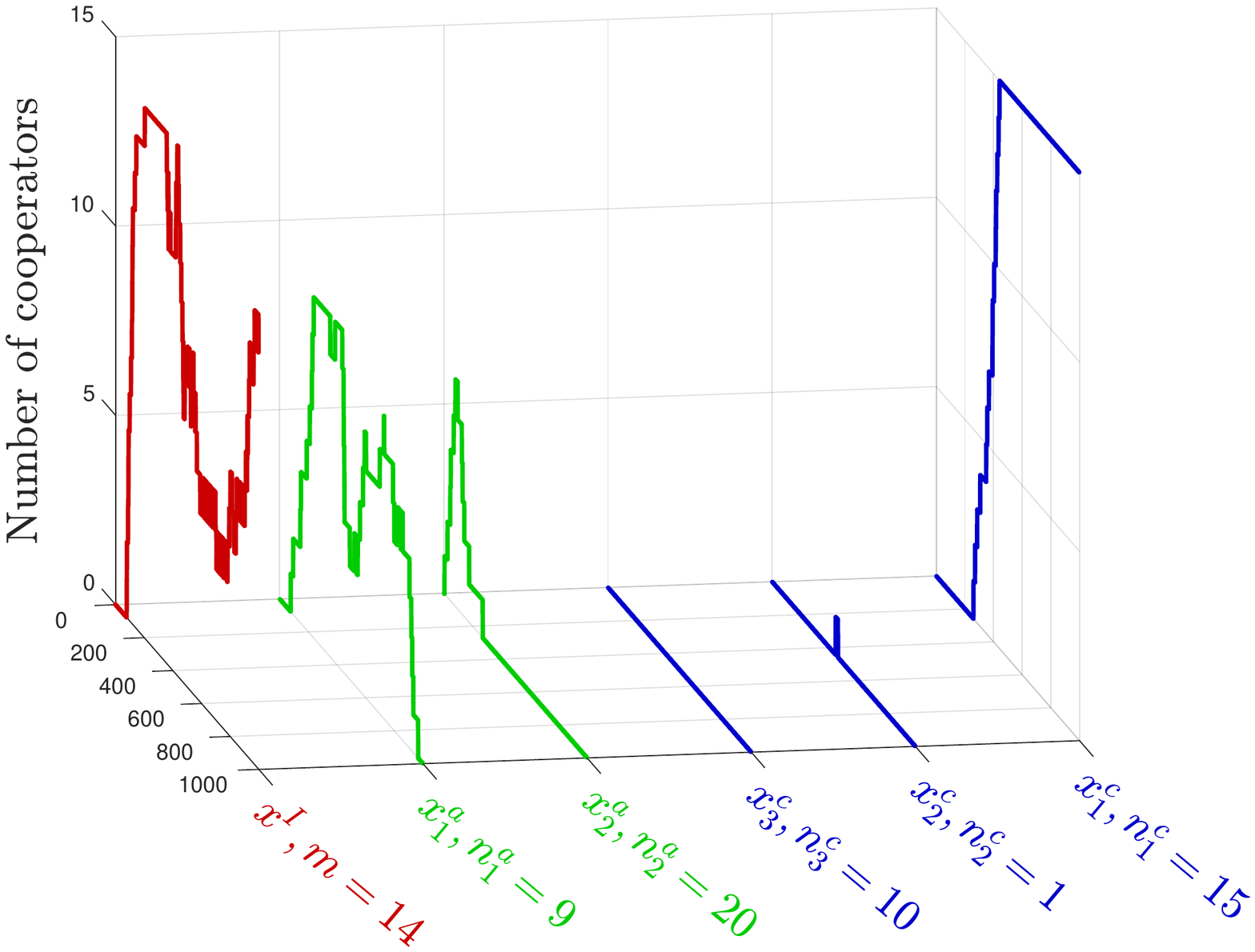}
            \caption{
            \textbf{Not reaching an equilibrium.}
            Even though all nonconformists of type $2$ and conformists of types $2$ and $3$ defect and all type-$1$ conformists cooperate in the long run, the population undergoes perpetual fluctuations as the imitators and type-$1$ nonconformists keep switching strategies. 
            }
            \label{fig:xf}
        \end{subfigure}
        \caption{The population dynamics under two different randomly generated activation sequences.}
        \label{fig:example}
    \end{figure*}
\end{example}

\begin{example}
    Consider a population of $68$ agents belonging to two anticoordinating and three coordinating types.
    The population distribution over the imitators, nonconformists, and conformists is $(\tr{m}, \tg{n_1^a}, \tg{n_2^a}, \tb{n_3^c}, \tb{n_2^c}, \tb{n_1^c}) = (\tr{14}, \tg{9}, \tg{20}, \tb{10}, \tb{5}, \tb{10})$.
    Among the $14$ imitators, $4$ of them are of type-$1$ anticoordinating.
    The payoffs are set so that they result in the following utility functions:
    \begin{align*}
        &C_1^a(n^C) = -\frac{20}{7} n^C+\frac{845}{7}, &&D_1^a(n^C) = \frac{18}{7} n^C-\frac{162}{7};\\
        &C_2^a(n^C) = \frac{8}{39} n^C+\frac{245}{13}, &&D_2^a(n^C) = \frac{16}{13} n^C+\frac{105}{13};\\
        &C_3^c(n^C) = \frac{4}{7} n^C-\frac{85}{7}, &&D_3^c(n^C) = -\frac{20}{21} n^C +\frac{347}{7};\\
        &C_2^c(n^C) = \frac{2}{3} n^C, &&D_2^c(n^C) = -\frac{4}{3}n^C+57;\\
        &C_1^c(n^C) = \frac{8}{19} n^C + \frac{615}{19}, &&D_1^c(n^C) = -\frac{16}{19} n^C+\frac{1107}{19}.
    \end{align*}
    The tempers are, hence, 
    \begin{equation*}
        (\tg{\tau_1^a}, \tg{\tau_2^a}, \tb{\tau_3^c}, \tb{\tau_2^c}, \tb{\tau_1^c}) = (\tg{26.5}, \tg{10.5}, \tb{40.5}, \tb{28.5}, \tb{20.5}).
    \end{equation*}
    This population does not admit any equilibrium, so it undergoes perpetual fluctuations (\cref{fig:xf2}) and confines
    the number of cooperators to the interval $[25,32]$ (\cref{fig:xf2_nc}).
    It is not difficult to see that the number of cooperators does not fall short of $21$ (nor exceeds $35$ in the long run).
    To see this, consider the time when the total number of cooperators is $26$. 
    Then based on the utility functions, a cooperator of anticoordinating type $1$ is earning the highest payoff. 
    Such a cooperator is guaranteed to exist since if all of the agents of anticoordinating type $1$ are defecting, the number of cooperators is at most $25$.
    This cooperator continues playing cooperation and earning the highest payoff when the number of cooperators decreases.
    So the cooperating imitators do not switch to defection.
    Type-$1$ nonconformists do not switch to defection either, because the number of cooperators is less than their tempers.
    So the cooperating conformists of type $2$ are the only ones who can switch to defection.
    However, there are only $5$ conformists of type $2$.
    On the other hand, the above argument holds when the number of cooperators is $24,23, 22$ or $21$. 
    Hence, the number of cooperators cannot decrease to $20$, after it reaches $26$.
    Similarly, one can verify that the number of cooperators cannot increase from $30$ to $35$.
    \begin{figure*}[h!]
            \begin{subfigure}[t]{0.49\textwidth}
            \centering
            \includegraphics[trim= {3cm 1cm 1cm 3cm}, clip, width=\textwidth]{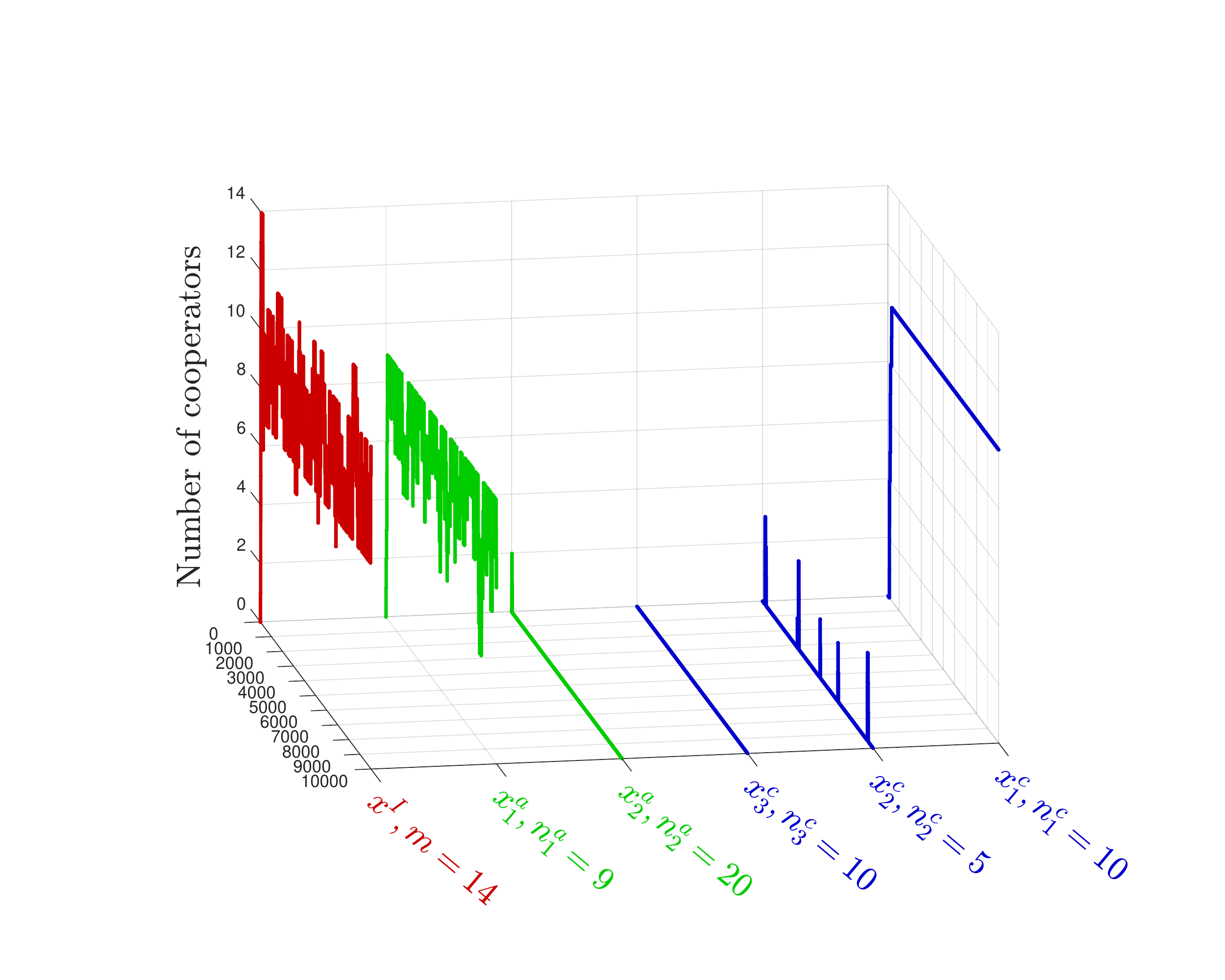}
            \caption{
            \textbf{Not reaching an equilibrium.}
            The dynamics do not possess an equilibrium. 
            The solution trajectories start from the state where all agents defect.
            All nonconformists of type $2$ and conformists of types $2$ and $3$ defect and all type-$1$ conformists cooperate in the long run, yet the population never reaches an equilibrium as the imitators and type-$1$ nonconformists keep switching strategies.}
            \label{fig:xf2} 
        \end{subfigure}
        \hfill
        \begin{subfigure}[t]{0.49\textwidth}
            \centering
            \includegraphics[trim ={1.5cm 5cm 0 5cm}, clip, width=\textwidth]{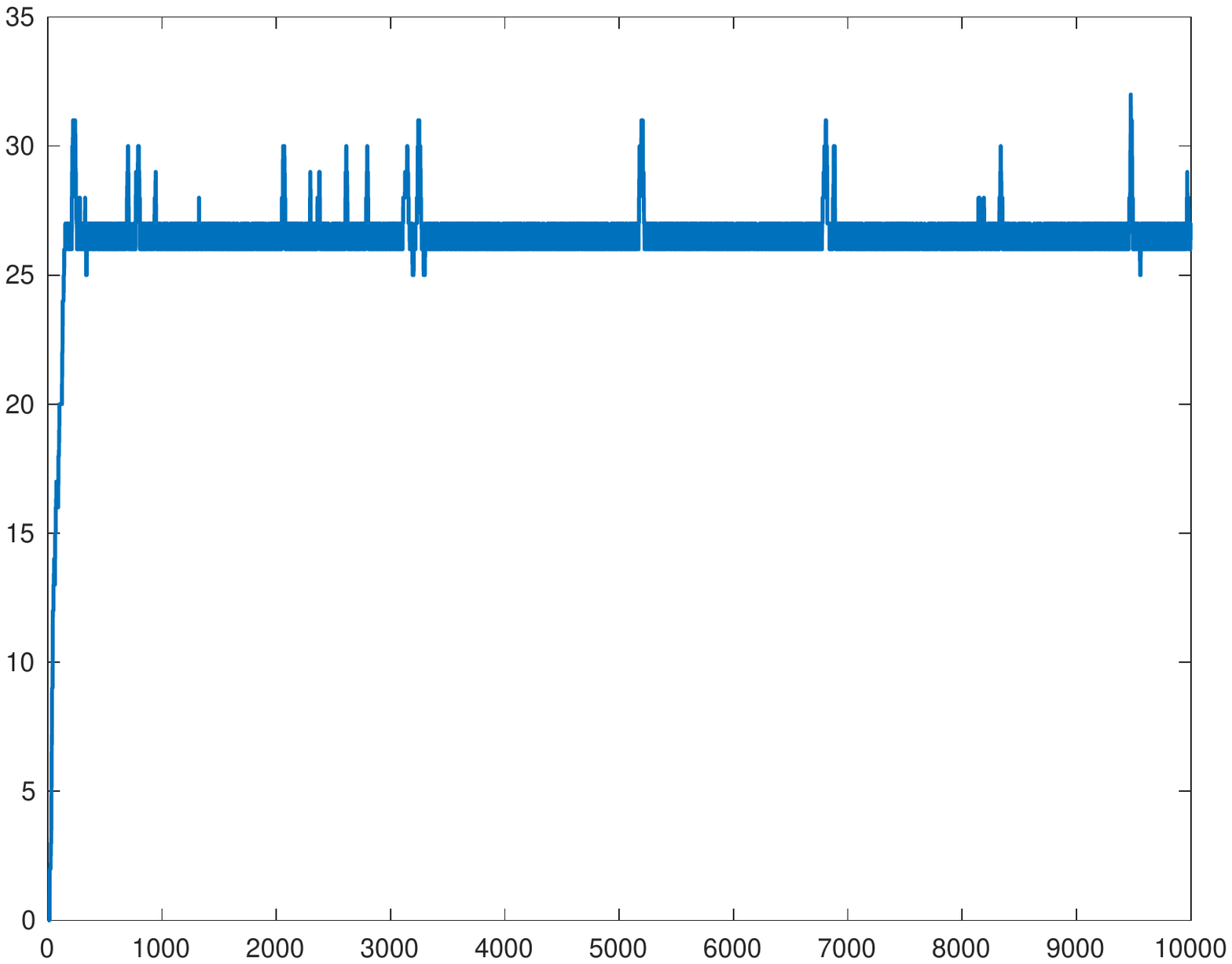}
            \caption{
            \textbf{The total number of cooperators.}
            The number fluctuates between $26$ and $30$.
            }
            \label{fig:xf2_nc}
        \end{subfigure}
        \caption{The population dynamics and the total number of cooperators under a randomly generated activation sequence.}
        \label{fig:example2}
    \end{figure*}
\end{example}

\section{Equilibria}    \label{sec:eq}
No agent would switch strategies at equilibrium.
Hence, from update rule \eqref{abr_update_rule}, if an nonconformist cooperates (resp. defects) at equilibrium, so will all others with higher (resp. lower) tempers, and according to \eqref{cbr_update_rule}, if a conformist cooperates (resp. defects), so will all others with lower (resp. higher) tempers.
Hence, we expect an equilibrium to be of the following form:
\begin{equation*}
    \x_{r, j_1, j'_1} 
    \triangleq (r, n_1^a, \ldots, n_{j_1}^a, 0,  \ldots, 0, 0,  \ldots, 0, n_{j'_1}^c, \ldots, n_1^c),
\end{equation*}
where $r \in \{0, \ldots, m\}$, $j_1\in \{0, \ldots, b\}$ and $j'_1\in \{0, \ldots, b'\}$.
Note that when $j_1 = 0$ (resp. $b$), all nonconformists defect (resp. cooperate) and when $j'_1= 0$ (resp. $b'$), all conformists defect (resp. cooperate). 
To simplify the analysis of these extreme cases, we define $\tau_{0}^a=\tau_{b'+1}^c$ as a number that is greater than $n$ and all other tempers and $\tau_{b+1}^a=\tau_{0}^c$ as a negative number that is smaller than all other tempers.
Then $\tau^a_j$ and $\tau^c_{j'}$ are well-defined for all $j\in\{0,\ldots,b+1\}$ and all $j'\in\{0,\ldots,b'+1\}$.

The goal of this section is to identify all equilibrium states that the population admits.
First, we prove in \cref{equil_form}, that all equilibria are of the form $\x_{r, j_1, j'_1}$.
Next, we find the necessary and sufficient conditions for states of the form $\x_{r, j_1, j'_1}$ to be an equilibrium in \cref{defect_equil}, \cref{coop_equil}, and \cref{mix_equil}.
Finally, we summarize the results in \cref{equil} and interpret them.

\subsection{Forms of equilibria}
\begin{lemma} \label{equil_form}
If $\x^*$ is an equilibrium with $r$ cooperating imitators, then there exist $j_1\in \{0, \ldots, b\}$ and $j'_1\in \{0, \ldots, b'\}$ such that $\x^*= \x_{r, j_1, j'_1}$.
\end{lemma}

\begin{proof}
Because $\tau_{0}^a = \tau_{b'+1}^c > n$ and $\tau_{b+1}^a = \tau_{0}^c < 0$, there exists $j_1\in [0, b]$ and $j'_1\in [0, b']$ such that $n^C(\x^*)\in (\tau_{j_1+1}^a, \tau_{j_1}^a)$ and $n^C(\x^*) \in (\tau_{j'_1}^c, \tau_{j'_1+1}^c)$.
We prove by contradiction that $x_i^a = n_i^a$ for all $i\in \left\{1, \ldots, j_1\right\}$.
Suppose on the contrary that $x_i^a < n_i^a$ for some $i \in \left\{1, \ldots, j_1\right\}$.
Let $\x(0) = \x^*$.
Under an activation sequence where the first active agent is a defecting nonconformist of type-$i$, she will switch to cooperation in view of \eqref{abr_update_rule} since 
\begin{equation*}
    n^C(0) = n^C(\x^*) < \tau_{j_1}^a \overset{\eqref{TemperOrder}}{\leq} \tau_i^a.
\end{equation*}
This contradicts $\x^*$ being an equilibrium.
Hence, $x_i^a = n_i^a$ for all $i\in \left\{1, \ldots, j_1\right\}$.
The remaining of the proof can be done similarly.
\end{proof}
\subsection{Necessary and sufficient condition}
For $\x_{r, j_1, j'_1}$ to be an equilibrium, cooperating (resp. defecting) best-responders must continue playing cooperation (resp. defection) upon activation.
Hence, it must hold that 
\begin{equation}\label{aTemperCond}
    \tau_{j_1+1}^a < n_{r, j_1, j'_1} < \tau_{j_1}^a,
\end{equation}
\begin{equation}\label{cTemperCond}
    \tau_{j'_1}^c < n_{r, j_1, j'_1} < \tau_{j'_1+1}^c,
\end{equation}
where $n_{r, j_1, j'_1}$ is the total number of cooperators at $\x_{r, j_1, j'_1}$:
\begin{equation}\label{def:n_(r, j_1, j'_1)}
    n_{r, j_1, j'_1} 
    \triangleq r+ \sum_{i=1}^{j_1} n_i^a+\sum_{i=1}^{j'_1} n_i^c.
\end{equation}
These conditions are necessary but not sufficient because of the presence of the imitators.
At equilibrium, either all imitators play the same strategy, resulting in $x^I =0$ or $m$, or there are both a defector and a cooperator earning the highest payoff.
This results in three types of equilibrium candidates:
\emph{(i) defection equilibrium candidate $\x_{0, j_1, j'_1}$}, 
\emph{(ii) cooperation equilibrium candidate $\x_{m, j_1, j'_1}$}, and
\emph{(iii) mixed equilibrium candidate $\x_{r, j_1, j'_1}$, $r\in \{1, \ldots, m-1\}$}.
For each of these candidates, we find separately the necessary and sufficient conditions to be an equilibrium. 
Denote the highest payoff earned by the cooperators of anticoordinating types $1, \ldots, j$ and coordinating types $ 1, \ldots, k$ when the total number of cooperators is $n^C$ by $C_{j, k}(n^C)$, i.e., 
$$C_{j, k}(n^C) = \sup \left\{\sup_{i\in \left\{1, \ldots, j\right\}} C_i^a(n^C), \sup_{i\in \left\{1, \ldots, k\right\}} C_i^c(n^C)\right\},$$
and denote the highest payoff earned by the defectors of anticoordinating types $j, \ldots, b$ and coordinating types $k, \ldots, b'$ when the total number of cooperators is $n^C$ by $D_{j, k} (n^C)$, i.e., 
$$D_{j, k} (n^C) = \sup \left\{\sup_{i\in \left\{j, \ldots, b\right\}} D_i^a(n^C), \sup_{i\in \left\{k, \ldots, b'\right\}} D_i^c(n^C)\right\}.$$

\begin{lemma} \label{defect_equil} 
    $\x_{0,j_1, j'_1}$ is an equilibrium if and only if \eqref{aTemperCond} and $\eqref{cTemperCond}$ hold for $r=0$, 
    and 
    \begin{equation}\label{C<D}
        C_{j_1, j'_1}(n_{0, j_1, j'_1}) \leq D_{j_1+1, j'_1+1}(n_{0, j_1, j'_1}).
    \end{equation}
\end{lemma}

\begin{proof}
    (Sufficiency) Let $\x(0) = \x_{0, j_1, j'_1}$.
    Then $n^C(0) \in (\tau_{j_1+1}^a, \tau_{j_1}^a)$ as \eqref{aTemperCond} holds for $r=0$.
    Therefore, nonconformists of types $1, \ldots, j_1$ (resp. $j_1+1, \ldots, b$) are cooperating (resp. defecting) and will not switch strategies upon activation according to \eqref{abr_update_rule} since $n^C(0) < \tau_{j_1}^a < \ldots < \tau_1^a$ in view of \eqref{TemperOrder}.
    Similarly, the conformists also stick to their strategies upon activation. 
    On the other hand, imitators are defecting and will not change their strategies upon activation  because
    \begin{equation*}
        \D(0)\geq D_{j_1+1, j'_1+1}(n_{0, j_1, j'_1}) \overset{\eqref{C<D}}{\geq} C_{j_1, j'_1}(n_{0, j_1, j'_1}) = \C(0).
    \end{equation*}
    So $\x(1) = \x(0)$ under any activation sequence, proving that $\x_{0, j_1, j'_1}$ is an equilibrium.
    
    (Necessity) Suppose $\x_{0, j_1, j'_1}$ is an equilibrium.
    Let $\x(0) = \x_{0, j_1, j'_1}$.
    Because nonconformists of type $j_1$ (resp. $j_1+1$) are cooperating (resp. defecting) and continue doing so upon activation, according to \eqref{abr_update_rule} it must hold that $n^C(0) <\tau_{j_1}^a$ (resp. $n^C(0)> \tau_{j_1+1}^a$).
    So \eqref{aTemperCond} holds for $r=0$.
    Similarly, one can prove that \eqref{cTemperCond} holds for $r=0$.
    Now we show by contradiction that \eqref{C<D} must be in force.
    Suppose, on the contrary, that
    \begin{equation}\label{defect_equil_eq2}
        C_{j_1, j'_1}(n_{0, j_1, j'_1}) > D_{j_1+1, j'_1+1}(n_{0, j_1, j'_1}).
    \end{equation}
    Because nonconformists of types $1, \ldots, j_1$ and conformists of types $1, \ldots, j'_1$ do not switch from cooperation to defection upon activation, it must hold that at time $0$, $C_i^a > D_i^a$ for any $i \in \{1, \ldots, j_1\}$, and $C_i^c > D_i^c$ for any $i \in \{1, \ldots, j'_1\}$.
    Hence, the followings hold when $n^C = n_{0, j_1, j'_1}$:
    \begin{equation*}
        C_{j_1, j'_1} \geq \sup_{i\in \left\{1, \ldots, j_1\right\}} C_i^a >\sup_{i\in \left\{1, \ldots, j_1\right\}} D_i^a,
    \end{equation*}
    \begin{equation*}
        C_{j_1, j'_1} \geq \sup_{i\in \left\{1, \ldots, j'_1\right\}} C_i^c >\sup_{i\in \left\{1, \ldots, j'_1\right\}} D_i^c,
    \end{equation*}
    which together with \eqref{defect_equil_eq2} result in $\C(0) = C_{j_1, j'_1}(n_{0, j_1, j'_1}) > D_{1,1}(n_{0, j_1, j'_1}) \geq \D(0)$.
    So under an activation sequence where the first active agent is an imitator, she will switch from defection to cooperation.
    This contradicts that $\x_{0,j_1, j'_1}$ is an equilibrium, completing the proof.
\end{proof}
So for a defection equilibrium candidate to be an equilibrium, it is necessary and sufficient that
\emph{(i)} the total number of cooperators is less than the tempers of cooperating nonconformists and defecting conformists,
\emph{(ii)} the total number of cooperators is greater than the tempers of defecting nonconformists and cooperating conformists, and
\emph{(iii)} a defecting best-responder earns the highest utility.
Similarly, we find the necessary and sufficient condition for $\x_{m, j_1, j'_1}$ to be an equilibrium.

\begin{lemma} \label{coop_equil}
    $\x_{m, j_1, j'_1}$ is an equilibrium if and only if
    \eqref{aTemperCond} and \eqref{cTemperCond} hold for $r=m$, and
    $$C_{j_1, j'_1}(n_{m, j_1, j'_1}) \geq D_{j_1+1, j'_1+1}(n_{m, j_1, j'_1}).$$
\end{lemma}

So the necessary and sufficient condition for a cooperation equilibrium candidate to be an equilibrium is the same as in \cref{defect_equil}, except that the second condition becomes ``\emph{(ii)} a cooperating best-responder earns the highest utility''.

\begin{lemma} \label{mix_equil}
    $\x_{r,j_1, j'_1}$, $r\in\{1, \ldots, m-1\}$, is an equilibrium if and only if
    \eqref{aTemperCond} and \eqref{cTemperCond} hold,
    and 
    \begin{equation}\label{C=D}
        C_{j_1, j'_1}(n_{r, j_1, j'_1}) = D_{j_1+1, j'_1+1}(n_{r, j_1, j'_1}).
    \end{equation}
\end{lemma}

\begin{proof}
    (Sufficiency) 
    Let $\x(0) = \x_{r, j_1, j'_1}$.
    It suffices to show that $\x(1) = \x(0)$ under any activation sequence.
    Similar to the proof of \cref{defect_equil}, it can be shown that because \eqref{aTemperCond} and \eqref{cTemperCond} hold, best-responders do not switch strategies upon activation. 
    To show that imitators do not change strategies either, we prove that $\C(0) = \D(0)$. 
    We first show that $\C(0) \geq \D(0)$.
    Because nonconformists of types $1, \ldots, j_1$ and conformists of types $1, \ldots, j_1$ are initially cooperating and continue doing so upon activation, similar to the proof of the necessity part in \cref{defect_equil}, it can be shown that
    \begin{equation*}
        \C(0) \geq C_{j_1, j'_1}(n_{r, j_1, j'_1}) >  \sup_{i\in \{1, \ldots, j_1\}} D_i^a(n_{r, j_1, j'_1}),
    \end{equation*}
    and
    \begin{equation*}
        \C(0) \geq C_{j_1, j'_1}(n_{r, j_1, j'_1}) >  \sup_{i\in \{1, \ldots, j_1\}} D_i^c(n_{r, j_1, j'_1}).
    \end{equation*}
    Together with \eqref{C=D}, these yield $\C(0) \geq D_{1,1}(n_{r, j_1, j'_1}) \geq \D(0)$.
    Similarly, we can show that $\D(0) \geq \C(0)$, which results in $\C(0) = \D(0)$ and completes the proof.
    
    (Necessity)
    Suppose that $\x_{r, j_1, j'_1}$ is an equilibrium.
    Similar to the proof of \cref{defect_equil}, it can be shown that \eqref{aTemperCond} and \eqref{cTemperCond} must be in force. 
    We prove by contradiction that \eqref{C=D} must also hold.
    Suppose on the contrary that $C_{j_1, j'_1}(n_{r, j_1, j'_1}) \neq D_{j_1+1, j'_1+1}(n_{r, j_1, j'_1})$.
    Then either of the following two cases holds:
    
    \emph{Case 1.} $C_{j_1, j'_1}(n_{r, j_1, j'_1}) > D_{j_1+1,j'_1+1}(n_{r, j_1, j'_1})$. 
    Then similar to the argument for the necessity part in \cref{defect_equil}, we can show that $\C(0) > \D(0)$.
    This leads to a contradiction since under an activation sequence where the first active agent is a defecting imitator, whose existence is guaranteed because $r \leq m-1$, she switches to cooperation. 
    
    \emph{Case 2.} $D_{j_1+1,j'_1+1} (n_{r, j_1, j'_1})> C_{j_1, j'_1}(n_{r, j_1, j'_1})$.
    Similar to the previous case, this also results in a contradiction, completing the proof.
\end{proof}
So the necessary and sufficient conditions for a mixed equilibrium candidate to be an equilibrium to be an equilibrium are \emph{(i)} the total number of cooperators is less (resp. greater) than the tempers of cooperating (resp. defecting) nonconformists and defecting (resp. cooperating) conformists, and
\emph{(ii)} both a defecting and cooperating best-responder earn the highest payoff.
\subsection{Summary of the results}
The following theorem follows from and summarizes Lemmas \ref{equil_form}, \ref{defect_equil}, \ref{coop_equil} and \ref{mix_equil}.
Denote the set of all equilibria by $\X^*$ and the set of states $\x_{r, j_1, j'_1}$ that are equilibria by $\mathcal{E}$, i.e., 
\begin{align}
    \mathcal{E} = \Big\{\x_{r, j_1, j'_1} \,|\,
    & r\in\{0, \ldots, m\}, j_1\in \{0, \ldots, b\}, j'_1\in\{0, \ldots, b'\},\nonumber\\
    &\tau_{j_1+1}^a < n_{r, j_1, j'_1} < \tau_{j_1}^a, \tau_{j'_1}^c < n_{r, j_1, j'_1} < \tau_{j'_1+1}^c, \label{cond_equil:eq1}\\
    &C_{j_1, j'_1}(n_{r, j_1, j'_1}) \geq D_{j_1+1, j'_1+1}(n_{r, j_1, j'_1}) \text{ if } r > 0 \label{cond_equil:eq2},\\
    &C_{j_1, j'_1}(n_{r, j_1, j'_1}) \leq D_{j_1+1, j'_1+1}(n_{r, j_1, j'_1}) \text{ if } r < m \label{cond_equil:eq3}\Big\}.
\end{align}
\begin{theorem} \label{equil}
    $\X^* = \mathcal{E}$.
\end{theorem}
This theorem provides an algorithmic approach to determine the set of all equilibria of the population dynamics.
For each combination of $(r, j_1, j'_1)$ where $r \in \{0, \ldots, m\}$, $j_1 \in \{0, \ldots, b\}$, and $j'_1 \in \{0, \ldots, b'\}$, we check Conditions \eqref{cond_equil:eq1}, \eqref{cond_equil:eq2} and \eqref{cond_equil:eq3}.
If all of them are satisfied, then $\x_{r, j_1, j'_1}$ is an equilibrium and is added to the set of equilibria.
Otherwise, we examine the next combination of $(r, j_1, j'_1)$.
Also note that from the examples in \cref{sec_examples}, we know that a single population may admit multiple equilibria with different number of cooperating imitators $r$, or may be not possess an equilibrium at all. 
\subsection{Interpretation of the results}  \label{subsec:interp}
%
At any equilibrium, cooperators form a group of agents who tend to cooperate when they are the only ones in the population who are cooperating (all outside the group are defecting).
We call such groups \emph{cooperation-preserving}.
If a cooperation-preserving group contains an nonconformist, the group size must be insufficient to make her switch strategies.
Moreover, if a cooperation-preserving group contains a conformist, the group size must be sufficient for her to keep playing her strategy.
As with imitators, if one belongs to the cooperation-preserving group, then an agent in the group must earn the highest utility when all inside the group are cooperating and all others, outside the group, are defecting.
Clearly, the cooperators at any equilibrium form a cooperation-preserving group.
However, a state with a cooperation-preserving group of cooperators may not be an equilibrium.
The condition is necessary but not sufficient because the defecting outsiders may switch strategies.
We, hence, define a cooperation-preserving group to be \emph{exclusive} if no agent outside the group tends to cooperate when they are all defecting. 
So the number of individuals in an exclusive cooperation-preserving group is too many for the outsiderish nonconformists to cooperate but not enough for the outsiderish conformists to cooperate.
On the other hand, if an exclusive cooperation-preserving group does not contain all of the imitators, then one of the outsiders must earn the highest utility when they all defect.
This argument results in the following proposition
\begin{proposition} \label{prop_interp}
    A state $x$ is an equilibrium if and only if the cooperators at $x$ form an exclusive cooperation-preserving group.
\end{proposition}

An exclusive cooperation-preserving group, hence, contains the maximum number of cooperative best-responders such that the group size exceeds the temper of conformists inside the group and nonconformists outside the group but falls short of the tempers of the nonconformists inside the group and conformists outside the group.
This is presented in \eqref{cond_equil:eq1}.
An exclusive cooperation-preserving group also contains the highest earner(s) if it includes an imitator, implied by \eqref{cond_equil:eq2}.
It, however, does not contain all highest earners if it does not include all imitators, which follows from \eqref{cond_equil:eq3}.


\section{Stability analysis}    \label{sec_stability}
For $\x\in\Z^{d}$, $d \in \N$, consider the first norm $\norm{\x}=\sum_{i=1}^{d} |x_i|$. 
An equilibrium $\x^*$ is \emph{stable} if for any $\epsilon > 0$, there exists $\delta > 0$ such that for any initial state $\x(0)$ with $\norm{\x(0)-\x^{*}} < \delta$, we have $\norm{\x(t)-\x^{*}} < \epsilon$ for all $t \geq 0$ under any activation sequence. 
Because any state $\x \in \X$ has $x_i \in \Z_{\geq0}$ for all $i \in \{1, \ldots,b+b'+1\}$, we require that $\delta > 1$.
Otherwise, $\x^*$ is always stable: for $0 < \delta \leq 1$, the only initial state that satisfies $\norm{\x(0)-\x^*} < \delta$ is $\x^*$ itself; hence, for any $\epsilon> 0$, we have $\norm{\x(t)-\x^*} = \norm{\x(0)-\x^*} = 0 <\epsilon$ for all $t\geq 0$ and for all activation sequences. 
Similarly, we require that $\epsilon > 1$; otherwise, no state will be stable.   
Indeed, to prove stability, we only need to consider all initial states $\x(0)$ with $\norm{\x(0)-\x^*}=1$. 
If starting from such states, $\norm{\x(t)-\x^*} \in \{0, 1\}$ for all $t\geq 0$ and under any activation sequence, then $\x^*$ is a stable equilibrium since for any $\epsilon > 1$, $\delta = 2$ satisfies the condition. 
Otherwise, $\x^*$ is unstable as the condition is violated for $\epsilon = 2$.
In the same way, the stability of equilibrium states has been studied under some discrete dynamics \cite{le2020heterogeneous, imitation, ramazi2018asynchronous}; however, none of these studies provide a general definition, which we do in the following:
\begin{definition}  \label{def_stability}
    Consider the time-dependent discrete dynamics $\x(t+1) = f(t,\x(t))$ where $\x\in\Z^d_{\geq0},d \in \N, f: \Z_{\geq0}\times\Z^d_{\geq0} \to \Z^d_{\geq0}$.
    Assume that the dynamics admit some  equilibrium state $\x^*$, i.e., $f(t,\x^*) = \x^*\forall t\in\mathbb{Z}_{\geq0}$.
    The equilibrium $\x^*$ is \emph{stable} if for any initial state $\x(0)$ with $\norm{\x(0)-\x^*}=1$, it holds that $\norm{\x(t)-\x^*} \leq 1 \, \forall t\in\mathbb{Z}_{\geq0}$.
\end{definition}

In our evolutionary game theory framework, this definition implies that if the population state starts from the closed radius-one sphere centered at the stable equilibrium, then the state may not leave the sphere, regardless of the activation sequence.
Namely, when all agents but one, say agent $p$, play the same strategies that they play at equilibrium, none would switch strategies (except for agent $p$ who may switch to her equilibrium strategy). 
In terms of best-responders, this implies that none would change their equilibrium strategies when the number of cooperators is one less or more than that at the equilibrium.
Hence, we expect the following to hold for a stable equilibrium $\x_{r, j_1, j'_1}, r \in \{0, \ldots, m\}$:
\begin{equation} \label{bestresponder_stab_cond1}
    n_{r, j_1, j'_1} \in (\tau_{j_1+1}^a+1, \tau_{j_1}^a-1),
\end{equation}
\begin{equation}\label{bestresponder_stab_cond2}
    n_{r, j_1, j'_1}\in (\tau_{j'_1}^c +1, \tau_{j'_1+1}^c-1).
\end{equation}
In the following three lemmas, we find the necessary and sufficient condition for defection, cooperation, and mixed equilibria to be stable.
To simplify the analysis, we assume that $n_i^a \geq 2$ for every $i \in\{1, \ldots, b\}$, $n_i^c \geq 2$ for every $i \in\{1, \ldots, b'\}$, and $m\geq 1$.

\begin{lemma}  \label{defect_stab}
    Consider $j_1 \in \{0, \ldots, b\}$, $j'_1 \in \{0, \ldots, b'\}$, where at least one of them is non-zero. 
    Then
    $\x_{0, j_1, j'_1}$ is a stable equilibrium if and only if \eqref{bestresponder_stab_cond1} and \eqref{bestresponder_stab_cond2} hold for $r=0$, 
    and for $n^C \in \{n_{0, j_1, j'_1}-1,n_{0, j_1, j'_1}, n_{0, j_1, j'_1}+1\}$,
    \begin{equation}\label{defect_stab_cond3}
        C_{j_1, j'_1}(n^C) \leq D_{j_1+1, j'_1+1}(n^C).
    \end{equation}
\end{lemma}
\begin{proof}
    (Sufficiency) 
    According to \cref{defect_equil}, $\x_{0,j_1, j'_1}$ is an equilibrium.
    Consider $\x(0)\in\X$ with $\norm{\x(0) - \x_{0, j_1, j'_1}} = 1$.
    At time $0$, all agents except one, say agent $p$, play the same strategies that they play at $\x_{0, j_1, j'_1}$.
    It suffices to show that $\x(1) \in \{\x(0), \x_{0,j_1, j'_1}\}$ under any activation sequence since inductively it can then be show that $\x(t) \in \{\x(0), \x_{0, j_1, j'_1}\}$ for any $t\geq 0$, implying that $\norm{\x(t)-\x_{0, j_1, j'_1}} \leq 1$.
    One of the following cases is in force:
    
    \emph{Case 1:}
    agent $p$ is an imitator.
    Then she must be cooperating; hence, $n^C(0) = n_{0, j_1,j'_1} +1$.
    In view of \eqref{bestresponder_stab_cond1} and \eqref{bestresponder_stab_cond2}, no best-responder switches strategies.
    Moreover, because \eqref{defect_stab_cond3} holds for $n^C = n_{0,j_1,j'_1}+1$, it can be shown, similar to the proof of \cref{defect_equil}, that $\D(0) \geq \C(0)$.
    This implies all imitators, except agent $p$, will stick to their current strategies.
    Thus, $\x(1) \in \{\x(0), \x_{0, j_1, j'_1}\}$ under any activation sequence. 
    
    \emph{Case 2:} agent $p$ is an nonconformist of some type $j \in \{1, \ldots, j_1\}$.
    Then agent $p$ must be defecting, making $n^C(0) = n_{0, j_1, j'_1} -1$.
    Hence, given \eqref{bestresponder_stab_cond1} and \eqref{bestresponder_stab_cond2}, no best-responder, except for agent $p$, will switch strategies.
    Similar to the first case, it can be shown that no imitator will switch strategies since she is defecting and \eqref{defect_stab_cond3} holds for $n_{0, j_1, j'_1}-1$.
    Hence, $\x(1) \in \{\x(0), \x_{0, j_1, j'_1}\}$ under any activation sequence.
    
    All other cases can be handled similarly:
    \emph{Case 3:} 
    agent $p$ is a conformist of some type $j \in \{1, \ldots, j'_1\}$;
    \emph{Case 4:}
    agent $p$ is an nonconformist of some type $j \in \{j_1+1, \ldots, b\}$;
    \emph{Case 5:}
    agent $p$ is a conformist of some type $j \in \{j'_1+1, \ldots, b'\}$.
    
    (Necessity) We prove by contradiction that \eqref{bestresponder_stab_cond1}, \eqref{bestresponder_stab_cond2} and \eqref{defect_stab_cond3} must hold.
    Suppose, on the contrary, that \eqref{bestresponder_stab_cond1} does not hold.
    If $n_{0, j_1, j'_1} \geq \tau_{j_1}^a-1$, then due to the assumption of having non-integer tempers, $n_{0, j_1, j'_1} > \tau_{j_1}^a-1$. 
    If $j_1 = 0$, then in view of \eqref{def:n_(r, j_1, j'_1)}, $n - m \geq n_{0, j_1, j'_1} > \tau_{j_1}^a-1 > n-1$, contradicting the assumption $m \geq 1$.
    Now consider the case $j_1 \geq 1$.
    For the initial state $\x(0)$ where all agents but an imitator play the same strategies that they play at $\x_{0, j_1, j'_1}$, under an activation sequence where the first active agent is an nonconformist of type $j_1$, she will switch from cooperation to defection since $n^C(0) = n_{0, j_1, j'_1}+1 > \tau_{j_1}$.
    This contradicts the stability of $\x_{0, j_1, j'_1}$ since $\norm{\x(1)-\x_{0, j_1, j'_1}} =2$.
    Similarly, $n_{0, j_1, j'_1} \leq \tau_{j_1+1}^a+1$ also leads to a contradiction.
    Therefore, \eqref{bestresponder_stab_cond1} must be in force.
    The same argument proves that \eqref{bestresponder_stab_cond2} holds.
    Now suppose on the contrary that \eqref{defect_stab_cond3} does not hold for some $n^C \in \{n_{0, j_1, j'_1}-1, n_{0, j_1, j'_1}, n_{0, j_1,j'_1}+1\}$.
    Based on \cref{defect_equil}, $n^C = n_{0, j_1, j'_1}$ satisfies \eqref{defect_stab_cond3} because $\x_{0, j_1,j'_1}$ is an equilibrium.
    Hence, one of the followings must hold:
    \begin{equation}\label{defect_stab:C>D:eq1}
        C_{j_1, j'_1}(n_{0, j_1, j'_1}-1) > D_{j_1+1, j'_1+1}(n_{0,j_1, j'_1}-1),
    \end{equation}
    \begin{equation}\label{defect_stab:C>D:eq2}
        C_{j_1, j'_1}(n_{0, j_1, j'_1}+1) > D_{j_1+1, j'_1+1}(n_{0,j_1, j'_1}+1).
    \end{equation}
    Either $j_1 > 0$ or $j'_1 >0$.
    We only prove the case with $j_1 > 0$ since the other can be handled similarly.
    If \eqref{defect_stab:C>D:eq1} holds, then consider an initial state where all agents play the same strategy that they play at $\x_{0, j_1, j'_1}$, except for an nonconformist of type $j_1$. 
    Then because $n^a_{j_1}\geq2$,
    there is still one more type-$j_1$ nonconformist who does not tend to switch to defection; otherwise, $\x_{0, j_1, j'_1}$ is not stable.
    Thus, $C^a_{j_1} \geq D^a_{j_1}$.
    It is also clear that $C^a_{i} \geq D^a_{i}$ for all $i\in\{1,\ldots, j_1-1\}$.
    Hence, in view of \eqref{defect_stab:C>D:eq1}, 
    $\C(0) > \D(0)$.
    Therefore, under an activation sequence where the first active agent is an imitator, she will switch to cooperation, yielding a contradiction.
    A similar argument shows that \eqref{defect_stab:C>D:eq2} also results in a contradiction, completing the proof.
\end{proof}

For the special case of $j_1=j'_1=0$, the state $\x_{0, 0, 0}$ can be shown to be stable if and only if it is an equilibrium and $n_{0, 0, 0} = 0 < \tau_1^c-1$, or equivalently $\tau^c_1 >1$.

Compared to the equilibrium condition for a defection equilibrium candidate (\cref{defect_equil}), the stability condition is slightly tighter.
It requires the number of cooperators to distance by at least one from the tempers of the benchmark types.
It also requires the defecting best-responders to earn the highest utility, also when the total number of cooperators changes by one.
Therefore, the defection equilibrium of a population in general ``is likely'' to be stable.
The following results can be proven similarly to \cref{defect_stab}.

\begin{lemma}  \label{coop_stab}
    Consider $j_1 \in \{0, \ldots, b\}$, $j'_1 \in \{0, \ldots, b'\}$ where at least one of $j_1 < b$ and $j'_1 < b$ hold.  
    Then
    $\x_{m, j_1, j'_1}$ is a stable equilibrium if and only if \eqref{bestresponder_stab_cond1} and \eqref{bestresponder_stab_cond2} hold for $r=m$, 
    and for $n^C \in \{n_{m, j_1, j'_1}-1,n_{m, j_1, j'_1}, n_{m, j_1, j'_1}+1\}$,
    \begin{equation*}
        C_{j_1, j'_1}(n^C) \geq D_{j_1+1, j'_1+1}(n^C).
    \end{equation*}
\end{lemma}

For the special case of $j_1=b$ and $j'_1=b$, $\x_{m, j_1, j'_1}$ is a stable equilibrium if and only if it is an equilibrium and $n_{m, b, b'} = n > \tau_b^c+1$.
\begin{lemma}  \label{mix_stab}
    Consider $j_1 \in \{0, \ldots, b\}$, $j'_1 \in \{0, \ldots, b'\}$. 
    Then for any $r \in \{1, \ldots, m-1\}$,
    $\x_{r, j_1, j'_1}$ is a stable equilibrium if and only if \eqref{bestresponder_stab_cond1} and \eqref{bestresponder_stab_cond2} hold,
    and for $n^C \in \{n_{r, j_1, j'_1}-1,n_{r, j_1, j'_1}, n_{r, j_1, j'_1}+1\}$,
    \begin{equation*}
        C_{j_1, j'_1}(n^C) = D_{j_1+1, j'_1+1}(n^C).
    \end{equation*}
\end{lemma}

The condition in the above lemma is strong since it requires $\x_{r-1, j_1, j'_1}$, $\x_{r, j_1, j'_1}$ and $\x_{r+1, j_1, j'_1}$ to be equilibria.
Namely, the mixed equilibrium must be surrounded by adjacent other equilibria in the state space.
It is, thus, generally ``unlikely'' for a mixed equilibrium to be stable.
The following theorem summarizes \cref{defect_stab}, \cref{coop_stab} and \cref{mix_stab}.
\begin{theorem}
    For $r \in \{0, \ldots, m\}$, $j_1 \in \{0, \ldots, b\}$ and $j'_1 \in \{0, \ldots, b'\}$ where $(r, j_1, j_1) \notin \{(0, 0, 0), (m, b, b')\}$, $\x_{r, j_1, j'_1}$ is a stable equilibrium if and only if \eqref{bestresponder_stab_cond1} and \eqref{bestresponder_stab_cond2} hold,
    and for $n^C \in \{n_{r, j_1, j'_1}-1,n_{r, j_1, j'_1}, n_{r, j_1, j'_1}+1\}$,
    \begin{equation*}
        C_{j_1, j'_1} \leq D_{j_1, j'_1} \text{ if } r <m \text{, and } C_{j_1, j'_1} \geq D_{j_1, j'_1} \text{ if } r > 0. 
    \end{equation*}
\end{theorem}
\section{Positively invariant sets}    \label{sec_invariant}
As seen in Example \ref{fluc_equil_eg}, the population may never settle down and undergo perpetual fluctuations, even though it admits an equilibrium.
In that case, the population is stranded in a set of states.
We define a non-empty set $\setS\subseteq\X$ to be \emph{positively invariant} if $\x(0)\in \setS$ implies $\x(t) \in \setS\forall t\geq 0$, under any activation sequence.
For simplicity, we refer to a positively invariant set as an \emph{invariant set}.
An invariant set is \emph{minimal} if it does not contain any other invariant set except its two trivial subsets, itself and the empty set.
Similar to \cite{imitation}, it is easy to show that starting from any initial condition, the population dynamics will reach a minimal invariant set $\setO$ in finite time and stay there hence.
If the minimal invariant set is a singleton, then it is basically an equilibrium state. 
Otherwise, it is a set of more than one states, each will be visited by the solution trajectory infinitely often:
\begin{lemma}[\cite{imitation}]\label{revisit_inv}
    Consider a minimal invariant set $\setO$.
    If $\x(0) \in \setO$, then for any state $\bar{\x} \in \setO$ and any $T\geq 0$, there exists some time $t\geq T$ such that $\x(t) = \bar{\x}$.
\end{lemma}

Our ultimate goal is to explicitly characterize all minimal invariant sets.
This, however, turns out to be an intricate problem. 
Instead, in each of the following subsections, we find necessary conditions for a set $\setO$ to be minimally invariant. 
Namely we define a set, say $\mathcal{A}$, satisfying the necessary conditions, and prove that $\setO\subseteq\mathcal{A}$.
This partially informs us about the form of $\setO$ and helps us to limit our search in finding the minimal invariant sets.
Ideally, if the conditions are sharp enough, we would have $\mathcal{A} = \setO$, which is not the case with this paper and remains an open problem. 
Instead, we find necessary and sufficient conditions for the set $\mathcal{A}$ to be invariant as then the existence of a minimal invariant subset $\setO\subseteq\mathcal{A}$ is guaranteed.

\subsection{Benchmark types}
Consider a invariant set $\setO$ and a solution trajectory $\x(t)$ with the initial condition $\x(0)\in\X$. 
The number of cooperators will then be confined to the interval $[\min_{\x\in\setO} n^C(\x), \max_{\x\in\setO} n^C(\x)]$, called a \emph{fluctuation interval}.
All nonconformists whose tempers are higher than the maximum number of cooperators in the fluctuation interval must be cooperating at any state in $\setO$ since, otherwise, they will switch from defection to cooperation upon activation and never change back, implying the existence of a non-trivial subset of $\setO$ which is invariant, leading to a contradiction.
Similarly, all nonconformists whose tempers are lower than the minimum number of cooperators in the fluctuation interval must be defecting.
We, hence, expect the existences of anticoordinating benchmark types $j_1, j_2 \in \{0, \ldots, b+1\}$, $j_1 < j_2$, such that all nonconformists of types $1, \ldots, j_1$ cooperate and all nonconformists of types $j_2, \ldots, b$ defect at every state in $\setO$.
By following a similar argument, we expect the existence of coordinating benchmark types $j'_1, j'_2 \in \{0, \ldots, b'+1\}$, $j'_1 < j'_2$, such that all conformists of types $1, \ldots, j'_1$ cooperate and all conformists of types $j'_2, \ldots, b'$ defect at every state in $\setO$.
Define $\setJ$ to be the set of all possible combinations of $(j_1, j_2, j'_2, j'_1)$, i.e., 
\begin{align*}
    \setJ \triangleq \Big\{(j_1, j_2, j'_2, j'_1) \,|\,
    & j_1\in \{0, \ldots, b\}, j_2 \in \{j_1+1, \ldots, b+1\}\\
    &j'_1\in \{0, \ldots, b'\}, j'_2 \in \{j'_1+1, \ldots, b'+1\} \Big\}.
\end{align*}
For $(j_1, j_2, j'_2, j'_1) \in \setJ$, define the set
\begin{align*}
    \mathcal{X}_{j_1, j_2, j'_2, j'_1} \triangleq
    \Big\{\x \in \X\,|\, 
    x_i^a = n_i^a&\ \forall i \in \left\{1, \ldots, j_1\right\}, \quad
    x_i^a = 0\ \forall i \in \left\{j_2, \ldots, b\right\}, \\
    x_i^c = n_i^c&\ \forall i \in \left\{1, \ldots, j'_1\right\}, \quad
    x_i^c = 0\ \forall i \in \left\{j'_2, \ldots, b'\right\}\Big\}.
\end{align*}
In what follows, we consider an arbitrary minimal invariant set $\setO$, and correspondingly define the minimum and maximum number of cooperators in the set, i.e.,  
$S = \min_{\x\in \mathcal{O}} n^C(\x)$, and $L = \max_{\x\in\mathcal{O}} n^C(\x)$, and the benchmark types 
\begin{equation}\label{O:j_1}
    \xi_1 \triangleq \max \left\{j \in \left\{0, b+1\right\}\,|\, \tau_j^a > L\right\},
\end{equation}
\begin{equation}\label{O:j_2}
    \xi_2 \triangleq \min \left\{j \in \left\{0, b+1\right\}\,|\, \tau_j^a < S\right\},
\end{equation}
\begin{equation}\label{O:j'_2}
    \xi'_2 \triangleq \min \left\{j \in \left\{0, b'+1\right\}\,|\, \tau_j^c > L\right\},
\end{equation}
\begin{equation} \label{O:j'_1}
    \xi'_1 \triangleq \max \left\{j \in \left\{0, b'+1\right\}\,|\, \tau_j^c < S\right\}.
\end{equation}
As a necessary condition for a set to be invariant, we prove that it must be included in $\mathcal{X}_{\xi_1, \xi_2, \xi'_2, \xi'_1}$.
\begin{proposition} \label{inv0}
    $\mathcal{O}\subseteq \mathcal{X}_{\xi_1, \xi_2, \xi'_2, \xi'_1}$.
\end{proposition}
\begin{proof}
We prove by contradiction.
Suppose on the contrary, that there exists $\x\in \mathcal{O}$ such that $\x\notin \mathcal{X}_{\xi_1, \xi_2, \xi'_2, \xi'_2}$. 
Then, one of the following cases holds:

\emph{Case 1:} 
$x_i^a < n_i^a$ for some $i\in \{1,\ldots, \xi_1\}$.
Consider the initial state $\x(0)=\x$.
Under an activation sequence where the first active agent is a defecting nonconformist of type $i$, she will switch to cooperation and never change back since for any $t\geq 0$, $\x(t) \in \setO$ implies that $n^C(t) \leq L \overset{\eqref{O:j_1}}{<} \tau_{\xi_1}^a \overset{\eqref{TemperOrder}}{\leq} \tau_i^a$.
This implies that $\x(t) \in \setO \backslash \{\x\}$ for all $t\geq 1$, contradicting \cref{revisit_inv}.

Similarly, we can prove that \emph{Case 2:} $x_i^c < n_i^c$ for some $i\in \{1,\ldots, \xi'_1\}$, \emph{Case 3:} $x_i^a > n_i^a$ for some $i\in \{\xi_2, \ldots, b\}$ and \emph{Case 4:} $x_i^c > 0$ for some $i\in \{\xi'_2, \ldots, b'\}$, all result in a contradiction, completing the proof.
\end{proof}

So when the population reaches a state in $\mathcal{O}$, all nonconformists of types $1, \ldots, \eta_1$ and conformists of types $1, \ldots, \eta'_1$ fix their strategies to cooperation, and all nonconformists of types $\eta_2, \ldots, m$ and conformists of types $\eta'_2, \ldots, m$ fix their strategies to defection.
We call these best-responders \emph{fixed} and refer to other best-responders as \emph{wandering}.
Now, we show that all wandering conformists defect when the number of cooperators reaches its minimum and cooperate when it reaches its maximum.
\begin{lemma}\label{coor_atmin}
    Consider a state $\x\in\mathcal{O}$.
    If $n^C(\x) = S$, then 
    $x_i^c = 0\ \forall i\in \left\{\xi'_1+1, \ldots, \xi'_2-1\right\}$, and
    if $n^C(\x) = L$, then
    $x_i^c = n_i^c\ \forall i\in \left\{\xi'_1+1, \ldots, \xi'_2-1\right\}$.
\end{lemma}
\begin{proof}
    First, we prove the first statement by contradiction.
    Suppose on the contrary there exists $\x \in \mathcal{O}$ such that $n^C(\x) = S$ and there is a cooperating conformist of type $i \in \left\{\xi'_1+1, \ldots, \xi'_2-1\right\}$.
    Consider the initial state $\x(0) =\x$.
    Under an activation sequence where a cooperating conformist of type $i$ is active at time $0$, she will switch to defection, because $S \overset{\eqref{O:j'_1}}{<} \tau_{\xi'_1+1}^c\overset{\eqref{TemperOrder}}{\leq} \tau_i^c$.
    This results in $n^C(1) = S-1 < S$, implying $\x(1)\notin \mathcal{O}$, contradicting the invariance of $\mathcal{O}$.
    The second part can be proved similarly.
\end{proof}

The wandering nonconformists do not necessarily form such a harmony at the extreme number of cooperators. 
Nevertheless, none of them may fix their strategies when the population dynamics are limited to $\setO$.
\begin{lemma}   \label{lemma_forEvery1}
    For every $i\in \left\{\xi_1+1, \ldots, \xi_2-1\right\}$, there exist $\x, \y \in\mathcal{O}$ such that $x_i^a \neq y_i^a$.
\end{lemma}
\begin{proof}
    We prove by contradiction.
    Suppose the contrary that there exist $i\in \left\{\xi_1+1, \ldots, \xi_2-1\right\}$ and $k \in \{0, \ldots, n_i^a\}$ such that for all $\x\in \mathcal{O}$, $x_i^a = k$. 
    One of the following two cases must be in force:
    
    \emph{Case 1}: $k< n_i^a$.
    Consider an initial state $\x(0) \in \setO$ with $n^C(0) = S$.
    Then $x_i^a(0)= k< n_i^a$ implies the existence of a defecting nonconformist of type $i$.
    Under an activation sequence where she is the first active agent, she will switch to cooperation because $S \overset{\eqref{O:j_2}}{<}\tau_{\xi_2-1}^a \overset{\eqref{TemperOrder}}{\leq} \tau_i^a$.
    This results in $x_i^a(1)=k+1$, which contradicts the assumption of $x_i^a$ being fixed to $k$.
    
    \emph{Case 2}: $k>0$.
    Similarly, we can show that this also leads to a contradiction, completing the proof.
\end{proof}

So when the population dynamics are in $\setO$, the wandering nonconformists of the same types are neither all fixed to cooperation nor all to defection.
\begin{lemma}
    For every $i\in \left\{\xi_1+1, \ldots, \xi_2-1\right\}$, there exists $\x,\y\in\mathcal{O}$ such that $0< x_i^a$ and $y_i^a < n_i^a$.
\end{lemma}
\begin{proof}
    We prove by contradiction. 
    Assume on the contrary, that no such pair of $\x$ and $\y$ exists. 
    Let $i\in \left\{\xi_1+1,\ldots, \xi_2-1\right\}$.
    Then either for every state $z\in\mathcal{O}$, $z_i^a = 0$ or for every state $z\in\mathcal{O}$, $z_i^a = n_i^a$, both contradicting \cref{lemma_forEvery1}. 
\end{proof}

We know that all states in a minimal invariant set $\setO$ belong to the set $\X_{j_1, j_2, j'_2, j'_1}$ for some $(j_1, j_2, j'_2, j'_1) \in \setJ$.
Yet this does not make $\X_{j_1, j_2, j'_2, j'_1}$ invariant. 
The following result states the necessary and sufficient condition for the set $\X_{j_1, j_2, j'_2, j'_1}$ to be invariant. 
Define 
\begin{equation*}
    \tau^{\max}_{j_2, j'_1} \triangleq \max\{\tau_{j_2}^a, \tau_{j'_1}^c\} 
    \quad\text{and}\quad
    \tau^{\min}_{j_1, j'_2} \triangleq \min \{\tau_{j_1}^a, \tau_{j'_2}^c\}.
\end{equation*}
\begin{proposition}\label{inv1}
    For $(j_1, j_2, j'_2, j'_1) \in \setJ$, $\mathcal{X}_{j_1, j_2, j'_2, j'_1}$ is a invariant set if and only if the following two hold:
    \begin{align}
        \tau^{\max}_{j_2, j'_1} < &\sum_{i=1}^{j_1} n_i^a +\sum_{i=1}^{j'_1} n_i^c, \label{inv1_eq2}\\
        &m+\sum_{i=1}^{j_2-1} n_i^a +\sum_{i=1}^{j'_2-1} n_i^c < \tau^{\min}_{j_1, j'_2}. \label{inv1_eq3}
    \end{align}
\end{proposition}
\begin{proof}
(Necessity) Suppose $\mathcal{X}_{j_1, j_2, j'_2, j'_1}$ is an invariant set. 
We first prove by contradiction that 
\begin{equation}\label{inv1:eq1}
    \sum_{i=1}^{j_1} n_i^a +\sum_{i=1}^{j'_1} n_i^c > \tau_{j_2}^a. 
\end{equation}
If on the contrary $\sum_{i=1}^{j_1} n_i^a +\sum_{i=1}^{j'_1} n_i^c \leq \tau_{j_2}^a$, then because of the non-integer tempers assumption, $\sum_{i=1}^{j_1} n_i^a +\sum_{i=1}^{j'_1} n_i^c < \tau_{j_2}^a$.
Hence, $j_2 \leq b$ since $\tau_{b+1}^a < 0$.
Consider 
$\x(0) = (0, n_1^a, \ldots, n_{j_1}^a, 0, \ldots, 0, 0, \ldots, 0, n_{j'_1}^c, \ldots,n^c_1)\in\mathcal{X}_{j_1, j_2, j'_2, j'_1}$.
If a type-$j_2$ nonconformist is initially active, she will switch from defection to cooperation because $n^C(0) =\sum_{i=1}^{j_1} n_i^a +\sum_{i=1}^{j'_1} n_i^c < \tau_{j_2}^a$. 
So $x_{j_2}^a(1)=1$, resulting in $\x(1) \notin \mathcal{X}_{j_1, j_2, j'_2, j'_1}$.
This contradicts that $\mathcal{X}_{j_1, j_2, j'_2,j'_1}$ is invariant. 
Similarly, it can be shown that $\sum_{i=1}^{j_1} n_i^a +\sum_{i=1}^{j'_1} n_i^c > \tau_{j'_1}^c$, which together with \eqref{inv1:eq1} yields \eqref{inv1_eq2}.
In the same way, by considering $\x(0) = (m, n_1^a, \ldots, n_{j_2-1}^a, 0, \ldots, 0, 0, \ldots, 0, n_{j'_2-1}^c, \ldots,n^c_1)$, we 
can prove \eqref{inv1_eq3}.

(Sufficiency)
Suppose \eqref{inv1_eq2} and \eqref{inv1_eq3} are in force.
It suffices to show that if $\x(0) \in \X_{j_1, j_2, j'_2, j'_1}$, then $\x(1) \in \X_{j_1, j_2, j'_2, j'_1}$ under any activation sequence.
Because $n^C(0) \leq \sum_{i=1}^{j_1} n_i^a+\sum_{i=1}^{j'_1} n_i^c < \tau^{\min}_{j_1, j'_2} \leq \tau_{j_1}^a$, nonconformists of types $1, \ldots, j_1$ do not switch strategies upon activation at time $0$ as they are already cooperating.
One can show the same for the remaining of the fixed agents, completing the proof.
\end{proof}

In essence, the necessary and sufficient invariance condition for the set $\X_{j_1, j_2, j'_2, j'_1}$ is that fixed best-responders do not switch strategies.
Note that this set is not tight enough as $\X_{\xi_1,\xi_2,\xi'_2,\xi'_2}$ may not be invariant, although $\setO\subset\X_{\xi_1,\xi_2,\xi'_2,\xi'_2}$.
\begin{corollary}
    For every $(j_1, j_2, j'_2, j'_1) \in \setJ$ that satisfies \eqref{inv1_eq2} and \eqref{inv1_eq3}, there exists a minimal invariant set $\setO\subseteq\X_{j_1, j_2, j'_2, j'_1}$.
\end{corollary}

\subsection{Minimum and maximum number of cooperators}
For a subset of $\X_{j_1, j_2, j'_2, j'_1}$ to be invariant, it is necessary that every state $\x$ in the subset to satisfy $n^C(\x) \in (\tau^{\max}_{j_2,j'_1}, \tau^{\min}_{j_1, j'_2})$.
Thus, we confine $\X_{j_1, j_2, j'_2, j'_1}$ to the subset 
\begin{align*}
    \setS_{j_1, j_2, j'_2, j'_1} \triangleq
    \Big\{\x\in \mathcal{X}_{j_1, j_2, j'_2, j'_1}\,|\,
    \tau^{\max}_{j_2, j'_1} < n^C(\x) < \tau^{\min}_{j_1, j'_2} \Big\}.
\end{align*}
It clearly holds that 
    $\setO \subseteq \setS_{\xi_1, \xi_2, \xi'_2, \xi'_1}$, resulting in $\tau^{\max}_{\xi_2, \xi'_1} < n^C(\x) < \tau^{\min}_{\xi_1, \xi'_2}$.
However, the types $\xi_1, \xi_2, \xi'_2,\xi'_1$, and in turn the thresholds $\tau^{\max}_{\xi_2, \xi'_1}$ and $\tau^{\min}_{\xi_1, \xi'_2}$, are defined based on the extremum number of cooperators $S$ and $L$.
So the inequality $\tau^{\max}_{\xi_2, \xi'_1} < n^C(\x) < \tau^{\min}_{\xi_1, \xi'_2}$ does not inform $\setO$ further than the trivial inequality $S < n^C(\x) < L$.
Instead, we find the necessary and sufficient condition for $\setS_{j_1, j_2, j'_2, j'_1}$ to be invariant.
Because for any $\x \in \X_{j_1, j_2, j'_2, j'_1}$, $\sum_{i=1}^{j_1} n_i^a + \sum_{i=1}^{j'_1} n_i^c \leq n^C(\x) \leq m+ \sum_{i=1}^{j_2-1} n_i^a + \sum_{i=1}^{j'_2-1} n_i^c$,
the following must hold for $\setS_{j_1, j_2, j'_2, j'_1}$ to be nonempty:
\begin{equation} \label{nonemptyS_cond1}
    \tau^{\max}_{j_2, j'_1} < m+ \sum_{i=1}^{j_2-1} n_i^a + \sum_{i=1}^{j'_2-1} n_i^c,
\end{equation}
\begin{equation}\label{nonemptyS_cond2}
    \sum_{i=1}^{j_1} n_i^a + \sum_{i=1}^{j'_1} n_i^c < \tau^{\min}_{j_1, j'_2}.
\end{equation}

\begin{lemma} \label{at_min2}
Given $(j_1, j_2, j'_2, j'_1) \in \setJ$ that satisfies \begin{equation} \label{lem_at_min2_eq1}
    \sum_{i=1}^{j_1} n_i^a + \sum_{i=1}^{j'_1} n_i^c  < \ceil{\tau^{\max}_{j_2, j'_1}}, 
\end{equation}
the set 
$\setS_{j_1, j_2, j'_2, j'_1}$ is invariant, only if the followings hold:
\begin{enumerate}
    \item at every $\x\in \setS_{j_1, j_2, j'_2, j'_1}$ with $n^C(\x) = \ceil{\tau^{\max}_{j_2, j'_1}}$, a cooperator earns the highest payoff;
    \item  $\tau_{j'_2-1}^c < \ceil{\tau^{\max}_{j_2, j'_1}} < \tau_{j_2-1}^a$. 
\end{enumerate}
\end{lemma}
\begin{proof}
    First we prove Part 1. by contradiction.
    Suppose on the contrary, there exists $\x\in \setS_{j_1, j_2, j'_2, j'_1}$ such that $n^C(\x) = \ceil{\tau^{\max}_{j_2, j'_1}}$ and all highest earners are defectors.
    Because of \eqref{lem_at_min2_eq1}, $n^C(\x) > \sum_{i=1}^{j_1} n_i^a + \sum_{i=1}^{j'_1} n_i^c$, implying the existence of a cooperator who is either an imitator, an nonconformist of type $i\in\{j_1+1,\ldots, b\}$ or a conformist of type $i\in\{j'_1+1,\ldots, b'\}$.
    Should the cooperator be an imitator, she would switch to defection if she becomes active, resulting in a state with the total number of cooperators less than $\ceil{\tau^{\max}_{j_2, j'_1}}$, contradicting the invariance of $\setS_{j_1, j_2, j'_2, j'_1}$.
    Should one of the other two cases hold, i.e., $x^a_i>0$ or $x^c_i>0$, then since $m>0$, the state $\y = \x +\bm{1}_1 - \bm{1}_{1+i}$ or $\y = \x +\bm{1}_1 - \bm{1}_{1+b+i}$ would also be in $\setS_{j_1, j_2, j'_2, j'_1}$, where $\bm{1}_j$ is the vector of all zeros except for the $j^{th}$ entry that is one. 
    In other words, we can modify $\x$ to obtain the state $\y\in \X_{j_1, j_2, j'_2, j'_1}$ with the same number of cooperators, resulting in $\y\in\setS_{j_1, j_2, j'_2, j'_1}$, but where instead of the best-responding cooperator, there is an imitating cooperator.
    Since the type of the imitator is the same as one of the best-responders, and the defectors among the best-responders were already earning the highest utility at $\x$, it can be concluded that defectors also earn the highest utility at $\y$.
    Hence, again the imitator would switch to defection if she becomes active, resulting in a contradiction.
    
    Next, we show by contradiction that $\ceil{\tau^{\max}_{j_2, j'_1}} < \tau_{j_2-1}^a$.
    Suppose on the contrary that $\ceil{\tau^{\max}_{j_2, j'_1}} \geq \tau_{j_2-1}^a$.
    Then, because of the non-integer temper assumption, $\ceil{\tau^{\max}_{j_2, j'_1}} > \tau_{j_2-1}^a$.
    According to $\sum_{i=1}^{j_1} n_i^a + \sum_{i=1}^{j'_1} n_i^c  < \ceil{\tau^{\max}_{j_2, j'_1}}$ and \eqref{nonemptyS_cond1}, there exists $\x\in \setS_{j_1, j_2, j'_2, j'_1}$ such that $n^C(\x) = \ceil{\tau^{\max}_{j_2, j'_1}}$ and an nonconformist of type $j_2-1$ is cooperating at $\x$.
    Consider the initial state $\x(0) = \x$.
    Under an activation sequence where the first active agent is a cooperating nonconformist of type $j_2-1$, she will switch to defection, which also results in a contradiction.
    Similarly, we can show that $\ceil{\tau^{\max}_{j_2, j'_1}} > \tau_{j'_2-1}^c$.
    The proof is then complete.
\end{proof}

\begin{lemma} \label{at_max2}
Given $(j_1, j_2, j'_2, j'_1) \in \setJ$ that satisfies 
\begin{equation*}
    m+ \sum_{i=1}^{j_1} n_i^a + \sum_{i=1}^{j'_1} n_i^c  > \floor{\tau^{\min}_{j_1, j'_2}},
\end{equation*} 
the set $\setS_{j_1, j_2, j'_2, j'_1}$ is invariant, only if the followings hold:
\begin{enumerate}
    \item at every $\x\in \setS_{j_1, j_2, j'_2, j'_1}$ with $n^C(\x) = \floor{\tau^{\min}_{j_1, j'_2}}$, a defector earns the highest payoff;
    \item  $\tau_{j_1+1}^a < \floor{\tau^{\min}_{j_1, j'_2}} < \tau_{j'_1+1}^c$.
\end{enumerate}
\end{lemma}
\begin{proof}
    The proof is similar to that of \cref{at_min2}.
\end{proof}

\begin{theorem}\label{inv2}
Given $(j_1, j_2, j'_2, j'_1) \in \setJ$, the set $\setS_{j_1, j_2, j'_2, j'_1}$ is invariant if and only if the following two conditions hold:
\begin{enumerate}
\item either $\sum_{i=1}^{j_1} n_i^a + \sum_{i=1}^{j'_1} n_i^c  \geq \ceil{\tau^{\max}_{j_2, j'_1}}$ or the followings hold:
    \begin{enumerate}
        \item at every $\x\in \setS_{j_1, j_2, j'_2, j'_1}$ with $n^C(\x) = \ceil{\tau^{\max}_{j_2, j'_1}}$, a cooperator earns the highest payoff;
        \item  $\tau_{j'_2-1}^c < \ceil{\tau^{\max}_{j_2, j'_1}} < \tau_{j_2-1}^a$, 
    \end{enumerate}
    \item either $m+ \sum_{i=1}^{j_1} n_i^a + \sum_{i=1}^{j'_1} n_i^c  \leq \floor{\tau^{\min}_{j_1, j'_2}}$ or the followings hold:
    \begin{enumerate}
        \item at every $\x\in \setS_{j_1, j_2, j'_2, j'_1}$ with $n^C(\x) = \floor{\tau^{\min}_{j_1, j'_2}}$, a defector earns the highest payoff;
        \item  $\tau_{j_1+1}^a < \floor{\tau^{\min}_{j_1, j'_2}} < \tau_{j'_1+1}^c$.
    \end{enumerate}
\end{enumerate}
\end{theorem}
\begin{proof}
    The necessity has been proved in \cref{at_min2} and \cref{at_max2}.
    For sufficiency, we show that if $\x(0) \in \setS_{j_1, j_2, j'_2, j'_1}$, then $\x(1) \in \setS_{j_1, j_2, j'_2, j'_1}$ under any activation sequence.
    Similar to the proof of \cref{inv1}, we can prove that $\x(1) \in \X_{j_1, j_2, j'_2, j'_1}$.
    Then it suffices to show that $n^C(1) \in (\tau_{j_2, j'_1}^{\max}, \tau_{j_1, j'_2}^{\min})$.
    To do so, we prove that if $n^C(0) = \ceil{\tau_{j_2, j'_1}^{\max}}$, then $n^C(1) \geq n^C(0)$, and similarly it can be shown that if $n^C(0) = \floor{\tau_{j_1, j'_2}^{\min}}$, then $n^C(1) \leq n^C(0)$.
    Because $\x(0) \in \setS_{j_1, j_2, j'_2, j'_1} \subseteq \X_{j_1, j_2, j'_2, j'_1}$, it holds that $\ceil{\tau_{j_2, j'_1}^{\max}} = n^C(0) \geq \sum_{i=1}^{j_1} n_i^a + \sum_{i=1}^{j'_1} n_i^c$.
    Then either of the two cases can happen:
    
    \emph{Case 1.1:} $\ceil{\tau_{j_2, j'_1}^{\max}} = \sum_{i=1}^{j_1} n_i^a + \sum_{i=1}^{j'_1} n_i^c$.
    Because $\x(1) \in \X_{j_1, j_2, j'_2, j'_1}$, we have $n^C(1) \geq  \sum_{i=1}^{j_1} n_i^a + \sum_{i=1}^{j'_1} n_i^c = n^C(0)$.
    
    \emph{Case 1.2:} $\ceil{\tau_{j_2, j'_1}^{\max}} > \sum_{i=1}^{j_1} n_i^a + \sum_{i=1}^{j'_1} n_i^c$.
    Then conditions 1.a and 1.b are in force.
    According to 1.a, a cooperator is earning the highest payoff, implying that no imitator will switch to defection.
    According to 1.b., nonconformists of types $j_1+1, \ldots, j_2-1$ and conformists of types $j'_1+1, \ldots, j'_2-1$ do not switch to defection either.
    The same holds with the other best-responders as $\x(0)$ and $\x(1)$ both belong to $\X_{j_1, j_2, j'_2, j'_1}$.
    So $n^C(1) \geq n^C(0)$ under any activation sequence, completing the proof.
\end{proof}

\subsection{Ordered summation of the wandering nonconformists}

Inspired by \cite{ramaziCharacter2020}, we expect the minimal invariant set to be a subset of the following set: 
\begin{align}
    \mathcal{I}_{j_1, j_2, j'_2, j'_1} = \Bigg\{\x\in \X_{j_1, j_2, j'_2, j'_1}\,\big|\, &\forall i\in\left\{j_1+1,\ldots, j_2-1\right\} \nonumber\\
    &\sum_{k=1}^{j'_1} n_k^c +\sum_{k=1}^{j_1} n_k^a +\sum_{k=i}^{j_2-1} x_k^a \leq \ceil{\tau_i^a} \label{r2l},\\
    & m+\sum_{k=1}^{j'_2-1} n_k^c +\sum_{k=1}^{j_1} n_k^a
    +\!\!\! \sum_{k=j_1+1}^{i}\!\!\! x_k^a+\!\!\sum_{k=i+1}^{j_2-1}\!\! n_k^a \geq \floor{\tau_i^a} \label{l2r}\Bigg\}.
\end{align}
For $\xi_1, \xi_2, \xi'_2, \xi'_1$ defined in \eqref{O:j_1}, \eqref{O:j_2}, \eqref{O:j'_2} and \eqref{O:j'_1}, we show the following result.
\begin{proposition} \label{supsetI}
    $\setO \subseteq \setI_{\xi_1, \xi_2, \xi'_2, \xi'_1}$.
\end{proposition}
\begin{proof}
    We prove by induction that for every $\x \in \setO$, \eqref{r2l} holds for all $i \in \{\xi_1+1, \ldots, \xi_2-1\}$.
    
    \emph{Step 1:} We show \eqref{r2l} for $i = \xi_2-1$.
    Let $\x\in \setO$.
    The result is trivial if $x_{\xi_2-1}^a = 0$ since 
    $\sum_{k=1}^{\xi'_1} n_k^c + \sum_{k=1}^{\xi_1} n_k^a \leq S < \tau_{\xi_2-1}^a < \ceil{\tau_{\xi_2-1}^a}.$
    So consider the case where $x_{\xi_2-1}^a > 0$.
    We prove by contradiction. 
    Suppose on the contrary, 
    \begin{equation}\label{supsetI:eq1}
        \sum_{k=1}^{\xi'_1} n_k^c + \sum_{k=1}^{\xi_1} n_k^a + x_{\xi_2-1}^a > \ceil{\tau_{\xi_2-1}^a}.
    \end{equation}
    Consider the initial state $\x(0) = \x$.
    Under an activation sequence where the first active agent is a cooperating nonconformist of type $\xi_2-1$, she will switch to defection since $n^C(0) \geq \sum_{k=1}^{\xi'_1} n_k^c + \sum_{k=1}^{\xi_1} n_k^a + x_{\xi_2-1}^a(0) \overset{\eqref{supsetI:eq1}}{>} \tau_{\xi_2-1}^a$.
    So $x_{\xi_2-1}^a(1) = x_{\xi_2-1}^a(0)-1$.
    Since $\x(0) \in \setO$, it holds that $\x(1) \in \setO$, which implies that $n^C(1) \geq \sum_{k=1}^{\xi_1} n_i^c+\sum_{k=1}^{\xi'_1} n_i^a + x_{\xi_2-1}^a(1) = \sum_{k=1}^{\xi_1} n_i^c+\sum_{k=1}^{\xi'_1} n_i^a + x_{\xi_2-1}^a(0) - 1 \overset{\eqref{supsetI:eq1}}{>} \tau_{\xi_2-1}^a$.
    Thus, a type-$(\xi_2-1)$ nonconformist will not switch from defection to cooperation at time $2$, resulting in $x_{\xi_2-1}^a(2) \leq x_{\xi_2-1}^a(1)< x_{\xi_2-1}^a(0)$. 
    By induction it can be shown that $x_{\xi_2-1}^a(t) < x_{\xi_2-1}^a(0)\forall t \geq 1$.
    Hence, $\x(t) \in \setO\backslash\{\x\}$ for all $t\geq 1$, contradicting the result of \cref{revisit_inv}.
    
    \emph{Step 2:} Suppose that \eqref{r2l} holds for some $i \in \{\xi_1+2, \ldots, \xi_2-1\}$ for every $\x\in\setO$.
    We show by contradiction that it also holds for $i-1$ for every $\x\in\setO$.
    Suppose the contrary that there exists $\x\in\setO$ such that
    \begin{equation}\label{supsetI:eq2}
        \sum_{k=1}^{\xi'_1} n_k^c + \sum_{k=1}^{\xi_1} n_k^a + \sum_{k=i-1}^{\xi_2-1} x_k^a >\ceil{\tau_{i-1}^a}.
    \end{equation}
    Then $x_{i-1}^a > 0$.
    Otherwise, $x_{i-1}^a = 0$, resulting in $\ceil{\tau_{i-1}^a} \overset{\eqref{supsetI:eq2}}{<} 
    \sum_{k=1}^{\xi'_1} n_k^c + \sum_{k=1}^{\xi_1} n_k^a + \sum_{k=i-1}^{\xi_2-1} x_k^a = \sum_{k=1}^{\xi'_1} n_k^c + \sum_{k=1}^{\xi_1} n_k^a + \sum_{k=i}^{\xi_2-1} x_k^a \overset{\eqref{r2l}}{\leq} \ceil{\tau_{i+1}^a}$, which contradicts \eqref{TemperOrder}.
    Consider the initial state $\x(0)= \x$. 
    Similar to Step 1, it can be shown that under activation sequences where the first active agent is a cooperating nonconformist of type $i-1$, $\sum_{k=i-1}^{\xi_2-1} x_k^a(t) < \sum_{k=i-1}^{\xi_2-1} x_k^a(0)$ for all $t\geq 1$, which also results in a contradiction.
    
    This completes the proof by induction for \eqref{r2l}.
    Similarly, we can show that for every $\x\in\setO$, \eqref{l2r} holds for any $i \in \{\xi_1+1, \ldots, \xi_2-1\}$, completing the proof.
\end{proof}

\section{Stochastic Stability}\label{sec:stochastic_stability}
So far, we have studied the agents' behavior under perfect decision makings, which requires sufficient computational capabilities and access to the necessary information. 
However, in practice, such assumptions are usually near impossible to meet. 
Individuals make mistakes in their decisions or need to experiment sub-optimal strategies in order to better understand their environment. So the real dynamics would be subjected to some \emph{perturbation}. To model this perturbation we suppose rather than choosing the strategy that is the best-response to the population or imitation of the most successful, the active individual may choose the opposite strategy with a small probability $\varepsilon>0$. To be more precise, let $s^{\mathfrak{UP}}_i(t)$ denote agent $i$'s strategy at time $t$ under the \emph{unperturbed} setting; that is, $s^{\mathfrak{UP}}_i(t) = s^{B}_i(t)$ or $s^I_i(t)$ based on whether agent $i$ is a best-responder or an imitator. 
Also, define $\bar{s}^{\mathfrak{UP}}_i(t)$ as the opposite strategy, i.e., $\{\bar{s}^{\mathfrak{UP}}_i(t)\} = \{C,D\} - \{s^{\mathfrak{UP}}_i(t)\}$. Upon activation at time $t$, agent $i$ at time $t+1$ plays the random strategy $s^{\mathfrak{P}}_i(t+1)$ defined by
\begin{equation*}
    s^{\mathfrak{P}}_i(t+1) = 
    \begin{cases}
        s^{\mathfrak{UP}}_i (t+1)& \text{if } Z^t = 0 \\
        \bar{s}^{\mathfrak{UP}}_i(t+1) & \text{if } Z^t = 1
    \end{cases},
\end{equation*}
where $Z^t$ is a Bernoulli random variable with mean $\varepsilon$. For each $t\in\mathbb Z_{\geq 0}$, $Z^t$ is supposed to be independent of $Z^0, \ldots, Z^{t-1}$ and every other happenings up to and including the time step $t$ (the previous and current states of the system, $i^0,\ldots,i^t$, etc.).

We are interested in the long-run behavior of the population dynamics in the presence of this type of perturbation. Under some conditions, not only the unperturbed dynamics but also the perturbed dynamics corresponding to every $\varepsilon>0$  can be represented by a Markov chain. When the perturbed dynamics are aperiodic and irreducible, for each $\varepsilon$, the long-run behavior is determined by the unique corresponding stationary distribution $\bm{\mu}^\varepsilon$. The limit of these stationary distributions, as $\varepsilon$ approaches zero, provides an approximation to $\bm{\mu}^\varepsilon$ for small values of $\varepsilon$. Those states of positive weight in this limiting distribution are considered to be \emph{stochastically stable} (see~\cite{foster1990stochastic,kandori1993learning} for the first emergences of this concept). We provide a rigorous definition of stochastic stability, that is fully explained in terms of Markov chains and their stationary distributions~\cite{young1993evolution,ellison2000basins}.

\begin{definition}{[Regular perturbation]}\label{def:evolution_with_noise}
Suppose that $P$ is a Markov transition matrix with a finite state space $\X$, i.e., for all $\x,\y\in\X$ and all $t\in\mathbb Z_{\geq0}$, the entry $P_{\x\y}$ is the probability of reaching $\y$ at time $t+1$, given that the state at time $t$ is $\x$. A \emph{regular perturbation} of $P$ is a family of Markov transition matrices $P^\varepsilon$ on $\X$ indexed by a parameter $\varepsilon\in(0,\bar{\varepsilon})$ for an appropriate value of $\bar\varepsilon>0$ such that
\begin{itemize}
\item{The Markov chain corresponding to each $P^\varepsilon$ is aperiodic and irreducible;}
\item{$P^\varepsilon$ is continuous in $\varepsilon$ with $\lim_{\varepsilon\to0}P^\varepsilon=P$;}
\item{there exists a \emph{cost function} $c_{\X}:\X\times\X\to[0,+\infty]$ such that for all pairs of states $\x, \y\in\X$, 
\begin{equation*}
\lim_{\varepsilon\to 0} P_{\x\y}^\varepsilon/\varepsilon^{c_{\X}(\x,\y)}
\end{equation*} 
exists and is strictly positive if $c_{\X}(\x,\y)<\infty$ (with $P_{\x\y}^\varepsilon=0$ for sufficiently small $\varepsilon$ if $c_{\X}(\x,\y)=\infty$).}
\end{itemize}
\end{definition}

We remind that $P_{\x\y}^\varepsilon$ is the probability of reaching $\y$ from $\x$ in one step under the perturbed Markov chain $P^\varepsilon$. So for each $\x, \y \in \X$, the value of $c_{\X}(\x,\y)$ shows how ``unlikely'' is the one-step transition from $\x$ to $\y$ under $P^\varepsilon$ for small values of $\varepsilon$, i.e., higher cost values imply more unlikely one-step transitions.  

Since the Markov chain corresponding to each $P^\varepsilon$ is aperiodic and irreducible, there exists a unique \emph{probability vector} $\bm{\mu}^\varepsilon$, i.e., a vector with non-negative components that sum to 1, which is the stationary distribution for $P^\varepsilon$. By definition, we have $\bm{\mu}^\varepsilon P^\varepsilon = \bm{\mu}^\varepsilon$. Indeed, the vector $\bm{\mu}^\varepsilon$ determines the long-run distribution of the chain independently of its initial distribution, i.e., for every initial probability distribution $\bm{\bar{\mu}}$, $\lim_{t\to+\infty}\bm{\bar{\mu}}(P^\varepsilon)^t = \bm{\mu}^\varepsilon$. By the continuity assumption of $P^\varepsilon$, we can define
\begin{align*}
\bm{\mu}^* = \lim_{\varepsilon\to 0}\bm{\mu}^\varepsilon.
\end{align*}
Then $\bm{\mu}^*$ is a probability vector and is used to define the stochastically stable states as follows.

\begin{definition}{[Stochastic stability]}\label{def:stoch_stable}
Consider $P^\varepsilon$ as a regular perturbation of $P$. A state $\x\in\X$ is \emph{stochastically stable} if $\bm{\mu}^*(\x)>0$. A non-empty subset of $\X$ is called \emph{stochastically stable} if its members are all stochastically stable.  The set of all stochastically stable states is called \emph{the maximal stochastically stable set}. 
\end{definition}

Since $\X$ is assumed to be finite, the maximal stochastically stable set is non-empty. Moreover, it is known that this set is a union of the \emph{recurrent classes} of the Markov chain corresponding to $P$ \cite[Theorem 4]{young1993evolution}.  Note that a recurrent class of a Markov chain with finite state space $\X$ is a non-empty subset $\Omega$ of $\X$ such that: 1) starting from $\Omega$, the chain stays there forever with probability $1$; 2) for any $\x,\y\in\Omega$ there exists a path of positive probability from $\x$ to $\y$.  Hence, the family of recurrent classes of $P$ is the same as the family of minimal invariant sets of the dynamics represented by $P$. 
\subsection{Mixed binary-type populations: set-up}\label{subsec: mixed_b_type_pop_set-up}

Inspired by the strong condition necessary for a mixed equilibrium to be stable (\cref{sec_stability}), we would like to investigate the problem of whether mixed or extreme equilibria are ``more likely'' to be stochastically stable. We focus on \emph{mixed binary-type populations}, i.e., populations consisting of two types: a conformist type and an nonconformist type, with some imitators from each type. We name those imitators with the same payoff matrix as that of the nonconformists (resp. conformists), the \emph{anticoordinating imitators} (resp. \emph{coordinating imitators}). In order to fulfill the assumptions of~\cref{def:evolution_with_noise}, we need to modify the model described in~\cref{sec:model}.

Let $m^a$ and $m^c$ represent respectively the total number of anticoordinating and coordinating imitators. Same as before, $n^a$ and $n^c$ are respectively the total number of nonconformists, and conformists. Define
\begin{equation*}
m = m^a+m^c\qquad\text{and} \qquad n = m^a+n^a+m^c+n^c, 
\end{equation*}
and assume $m^a, n^a, m^c, n^c\geq 1$. The tempers of the nonconformists and conformists are respectively non-integer values $\tau^a$, and $\tau^c$. The population state is represented by the frequency of cooperating agents in each of the four subpopulations, i.e, anticoordinating imitators, nonconformists, coordinating imitators, and conformists:
\begin{equation*}
\x = (x_1^I, x^a, x_2^I, x^c).
\end{equation*}
Unlike the setting in~\cref{sec:model}, here we use this different representation that includes anticoordinating and coordinating imitators separately since we need the  population dynamics to be described as a Markov chain on the state space $\X$. More specifically, it is necessary that for any two states $\x$ and $\y$, the probability of transition form $\x$ to $\y$ depends only on these states. Consider, for instance, arbitrary $m^a$, $n^a$, $m^c$, and $n^c$ such that $m^a = m^c$. The probability of transition from $(m^a, n^a, 0, n^c)$ to $(m^a, n^a, 1, n^c)$ might be different from the probability of transition from $(0, n^a, m^c, n^c)$ to $(1, n^a, m^c, n^c)$, for the active utility lines at $(m^a, n^a, 0, n^c)$ and $(0, n^a, m^c, n^c)$ are not the same. However, both states have the same number of cooperating imitators. This is also the case for the destination states $(m^a, n^a,1, n^c)$ and $(1, n^a, m^c, n^c)$. Hence, distinguishing the imitators in the state is required for having a Markov chain.

 On the other hand, the Markov property together with the necessity of irreducibility in the perturbed case imposes some constraints on the activation sequence. The transition probabilities must depend only on the source and destination states, and this dependency may not change over time. Moreover, in the perturbed dynamics, every state must be accessible from any other state through a path of positive probability.  So we impose the following assumption.

\begin{assumption}\label{assump:activation_sequence}
Suppose that for $t\in\mathbb Z_{\geq0}$, the random variable $I^t$ represents the active agent at time $t$. The $i^t$ in~\cref{sec:model} will then be a realization of $I^t$. We assume $I^0, I^1,\ldots$ are independent and identically distributed according to a positive probability vector $p = (p_1,\ldots,p_n)$, i.e., for $i\in\{1,\ldots,n\}$ and $t\in\mathbb Z_{\geq0}$, $\mathbb P[I^t = i]=p_i>0$. Moreover, inside each subpopulation, $p$ is constant; that is, $p_i$ is the same for all agents $i$ that are conformists, as it is the same for all nonconformists, for all coordinating imitators, and for all anticoordinating imitators.

\end{assumption}

Note that the positivity assumption guaranties the irreducibility necessary for the perturbed dynamics, and assuming the ``uniformity'' inside subpopulations is necessary for the Markov property.

For the explained regular perturbation, the cost function $c_{\X}$ is zero exactly for those $(\x,\y)$ that $\y$ can be obtained from $\x$ in one step without trembling, i.e., when $P_{\x\y}>0$ . Moreover, $c_{\X}(\x,\y) = 1$ when exactly one mistake is needed to go from $\x$ to $\y$ in one step, i.e., for those $(\x,\y)$ that $P_{\x\y}=0$ and $P_{\x\y}^\varepsilon>0$ when $\varepsilon$ is small enough. Otherwise, $c_{\X}(\x, \y)=+\infty$. 

In the next subsection, we review some related notions and results we need. For a more comprehensive reading including richer results on identifying stochastically stable states, we refer the reader to \cite{young1993evolution,ellison2000basins,rozen2008dual}.

\subsection{Preliminaries}\label{subsec:preliminaries}
In this subsection, we abstract away the game-theoretic details, and only focus on the ongoing dynamics, represented as Markov chains. As noted before, it is proved in~\cite[Theorem 4]{young1993evolution} that the maximal stochastically stable set of a regular perturbation of a Markov transition matrix $P$ is a non-empty union of the recurrent classes of $P$. Intuitively, a stochastically stable recurrent class can be reached relatively quickly from outside the class, and the chain leaves the class with a relatively small probability. This intuition can be rigorously explained using the cost function in~\cref{def:evolution_with_noise}. Note that transitions between different states can occur trough paths of longer lengths, while the function $c_{\X}$ only describes the chance of one-step transitions. Using this function, we define an optimal path between any two states, or more generally, any two non-empty sets. For non-empty subsets $\mathcal U, \mathcal V\subseteq\X$, a \emph{path} from $\mathcal U$ to $\mathcal V$ in $\X$ is an ordered finite tuple $(\z_1,\z_2,\ldots,\z_T)$ of distinct states of $\X$, where $\z_1\in \mathcal U$, $\z_T\in \mathcal V$, and $\z_t\not\in \mathcal V$ for $2\leq t\leq T-1$. To any path $(\z_1,\z_2,\ldots,\z_T)$, assign the \emph{cost} value
\begin{align*}\label{eq:path_cost}
c(\z_1,\z_2,\ldots,\z_T) = \sum_{i=1}^{T-1}c_{\X}(\z_i,\z_{i+1}),
\end{align*}
and define the \emph{transition cost} from $\mathcal U$ to $\mathcal V$, denoted by $c(\mathcal U,\mathcal V)$, to be the minimum cost of a path from $\mathcal U$ to $\mathcal V$:
\begin{align}\label{eq:generalized_cost}
c(\mathcal U,\mathcal V) = \min_{(\z_1,\ldots,\z_T)\in\Pi(\mathcal U,\mathcal V)}c(\z_1,\ldots,\z_T), 
\end{align}
where $\Pi(\mathcal U,\mathcal V)$ is the set of all paths from $\mathcal U$ to $\mathcal V$ in $\X$. By irreducibility of $P^\varepsilon$s, there exist paths of positive probability and finite length from an arbitrary state to any other state in the dynamics. Hence, $c(\mathcal U,\mathcal  V)$ is finite for non-empty $\mathcal U$ and $\mathcal  V$. 
In what follows, we use states instead of singletons for convenience, e.g., $c(\x,\y)$, that should not be confused with $c_{\X}(\x,\y)$, and $\Pi(\x,\mathcal U)$ are used respectively in place of $c(\{\x\},\{\y\})$ and $\Pi(\{\x\},\mathcal U)$.
\begin{remark}
Note that in the mixed binary-type population dynamics, for any non-empty $\mathcal U, \mathcal V\subseteq \mathcal X$, the cost function $c(\mathcal U,\mathcal V)$ is the minimum number of mistakes necessary to have a transition from an element of $\mathcal U$ to an element of $\mathcal V$.
\end{remark}

Using~\eqref{eq:generalized_cost}, we define a complete weighted digraph $\mathcal G(P^\varepsilon)$, with the vertex set consisting of all recurrent classes of $P$, denoted by $\Omega_1,\ldots,\Omega_k$. For $i\neq j$, we assign the weight $c(\Omega_i, \Omega_j)$ to the edge $\Omega_i\longrightarrow\Omega_j$. For a recurrent class $\Omega$, by an \emph{$\Omega$-tree} we mean a spanning subtree of $\mathcal G(P^\varepsilon)$ rooted in $\Omega$.
\begin{definition}\label{def:i-tree}
For the regular perturbation $P^\varepsilon$ of $P$, suppose $\Omega$ is a recurrent class of $P$. An \emph{$\Omega$-tree} is a spanning subtree of $\mathcal G(P^\varepsilon)$ such that for any recurrent class $\Omega'\neq\Omega$, there is a unique directed path in this subtree from $\Omega'$ to $\Omega$. 
\end{definition}

Denote the set of all $\Omega$-trees for a given $\Omega$ by $\mathcal T_{\Omega}$, and define $\gamma(\Omega)$ as the minimum (additive) weight among all elements of $\mathcal T_{\Omega}$; that is,
\begin{align}\label{eq:gamma}
\gamma(\Omega) = \min_{\tau\in\mathcal T_{\Omega}}\sum_{\Omega_i\to\Omega_j\in\tau}c(\Omega_i,\Omega_j).
\end{align}
The following has been proven in \cite{young1993evolution}.

\begin{proposition}\label{prop:young_class_stability}
Consider the regular perturbation $P^\varepsilon$ and let $\Omega_1, \ldots, \Omega_k$ be the recurrent classes of $P$. A recurrent class $\Omega$ is stochastically stable if and only if 
\begin{align*}\label{eq:argmin_tree}
\Omega\in\argmin_{\Omega_j\in\{\Omega_1,\ldots,\Omega_k\}}\gamma(\Omega_j).
\end{align*}
\end{proposition}

\cref{prop:young_class_stability} can be utilized to identify the maximal stochastically stable set. However, since the number of rooted spanning subtrees of $\mathcal G(P^\varepsilon)$ grows exponentially with the number of recurrent classes, more efficient methods than finding a minimum $\Omega$-tree for every recurrent class $\Omega$ are needed. In~\cite{ellison2000basins}, Ellison defines other quantities helpful in finding stochastically stable sets. One is the \emph{radius} of a recurrent class $\Omega$, which is a measure of the ``persistence'' of $\Omega$ once reached.

 \begin{definition}\label{def:basin_radius}
 Suppose $\Omega$ is a recurrent class of a Markov chain on $\X$ with transition matrix $P$. The \emph{basin of attraction} of $\Omega$ is the set of initial states from which the Markov chain reaches $\Omega$ with probability one, i.e.,
 \begin{equation*}
 \mathcal D(\Omega)=\left\{\x\in \X:\mathbb P [\exists T\in\mathbb N\  \x(T) \in\Omega |\x(0)=\x]=1\right\},
 \end{equation*}
 where $\x(t)$ denotes the state of the chain at time $t\in\mathbb Z_{\geq0}$.
 For the regular perturbation $P^\varepsilon$ of $P$, the \emph{radius} of the basin of attraction of $\Omega$, i.e., $R(\Omega)$, is the \emph{cost of leaving} $\mathcal D(\Omega)$ starting from $\Omega$; that is,
 \begin{align*}
 R(\Omega) = c(\Omega, \X\setminus \mathcal D(\Omega)).
 \end{align*}
 \end{definition}
The following lemma, that will be proved in~\cref{app:modified_cost}, is based on~\cite[Theorem 3]{ellison2000basins}.
\begin{lemma}\label{lem:c_instead_of_c*}
 Let $P^\varepsilon$ be a regular perturbation of $P$. For a recurrent class $\Omega$ of $P$, and a state $\x\not\in\Omega$, if we have $R(\Omega)>c(\x,\Omega)$, then $\bm{\mu}^*(\x)=0$. If $R(\Omega)=c(\x,\Omega)$, then $\bm{\mu}^*(\x)>0$ implies $\bm{\mu}^*(\Omega)>0$.
\end{lemma}

 \subsection{Mixed binary-type populations: results}
 
Since this work is focused on comparing the stochastic stability of mixed and extreme equilibria of mixed binary-type populations described in~\cref{subsec: mixed_b_type_pop_set-up}, we first characterize the general forms of equilibrium points of this model. Same as before, we call an equilibrium mixed if the number of its cooperating imitators belongs to $\{1,\ldots,m-1\}$. Other equilibria are called extremes. According to~\cref{mix_equil}, any mixed equilibrium takes either of the forms $(x_1^I, n^a, x_2^I, 0)$ where
\begin{equation*}
C^a(x_1^I+n^a+x_2^I) = D^c(x_1^I+n^a+x_2^I),
\end{equation*}
or $(x_1^I,0,x_2^I, n^c)$ where
\begin{equation*}
C^c(x_1^I+x_2^I+n^c) = D^a(x_1^I+x_2^I+n^c).
\end{equation*}
As a result, there exists at most two distinct values of $r\in\{1,\ldots, m-1\}$ that correspond to the frequency of cooperating imitators at a mixed equilibrium. Note that the multiplicity of mixed equilibria with a specific number of cooperating imitators is either zero or at least two. If, say, $(x_1^I, n^a, x_2^I, 0)$ is a mixed equilibrium and $x_1^I+x_2^I = r$, then any state $(z_1^I, n^a, z_2^I, 0)$, where $z_1^I+z_2^I = r$, $z_1^I\in\{0,\ldots,m^a\}$, and $z_2^I\in\{0,\ldots,m^c\}$ is a mixed equilibrium. We call these mixed equilibria \emph{congruent}.  The situation is similar with mixed equilibria of type $(x_1^I,0,x_2^I, n^c)$. According to~\cref{defect_equil,coop_equil}, every extreme equilibrium takes one of the forms $(0, 0, 0, 0)$, $(0, n^a, 0, 0)$, $(m^a, n^a, m^c, 0)$, $(0, 0, 0, n^c)$, $(m^a, 0, m^c, n^c)$, or $(m^a, n^a, m^c, n^c)$, depending on the parameters of the population and utility functions. 

We will show that the stochastic stability of a mixed equilibrium transmits to that of its ``neighboring'' extreme equilibrium defined as follows. 
\begin{definition}\label{def:corresponding_extreme}
Given the mixed equilibrium $(x_1^I, n^a, x_2^I, 0)$ (resp. $(x_1^I, 0, x_2^I, n^c)$), we define its corresponding extreme state as $(0, n^a, 0, 0)$ (resp. $(m^a, 0, m^c, n^c)$).
\end{definition}

Depending on the dominant utility functions around a mixed equilibrium, its corresponding extreme state may be an equilibrium. If this is the case, ~\cref{thm:stochastically_stable_equilibria} states that the stochastic stability of the mixed equilibrium implies the stochastic stability of its corresponding extreme state. In~\cref{fig:schematic_utility_lines} the utility lines in dynamics with non-congruent mixed equilibria are depicted. In this figure, it is supposed that the corresponding extreme state of each mixed equilibrium is  an equilibrium. The position of utility lines around each mixed equilibrium is the same as what is shown in~\cref{fig:schematic_utility_lines} even if we relax the constraint of the existence of the other non-congruent mixed equilibrium. Note that $\tau^a\neq\tau^c$ and $0<\tau^a,\tau^c<n$ necessarily hold in such dynamics. 
        \begin{figure*}[h!]
            \begin{subfigure}[t]{0.49\textwidth}
            \centering
            \includegraphics[trim= {2.5cm 6cm 2cm 6cm}, clip, width=\textwidth]{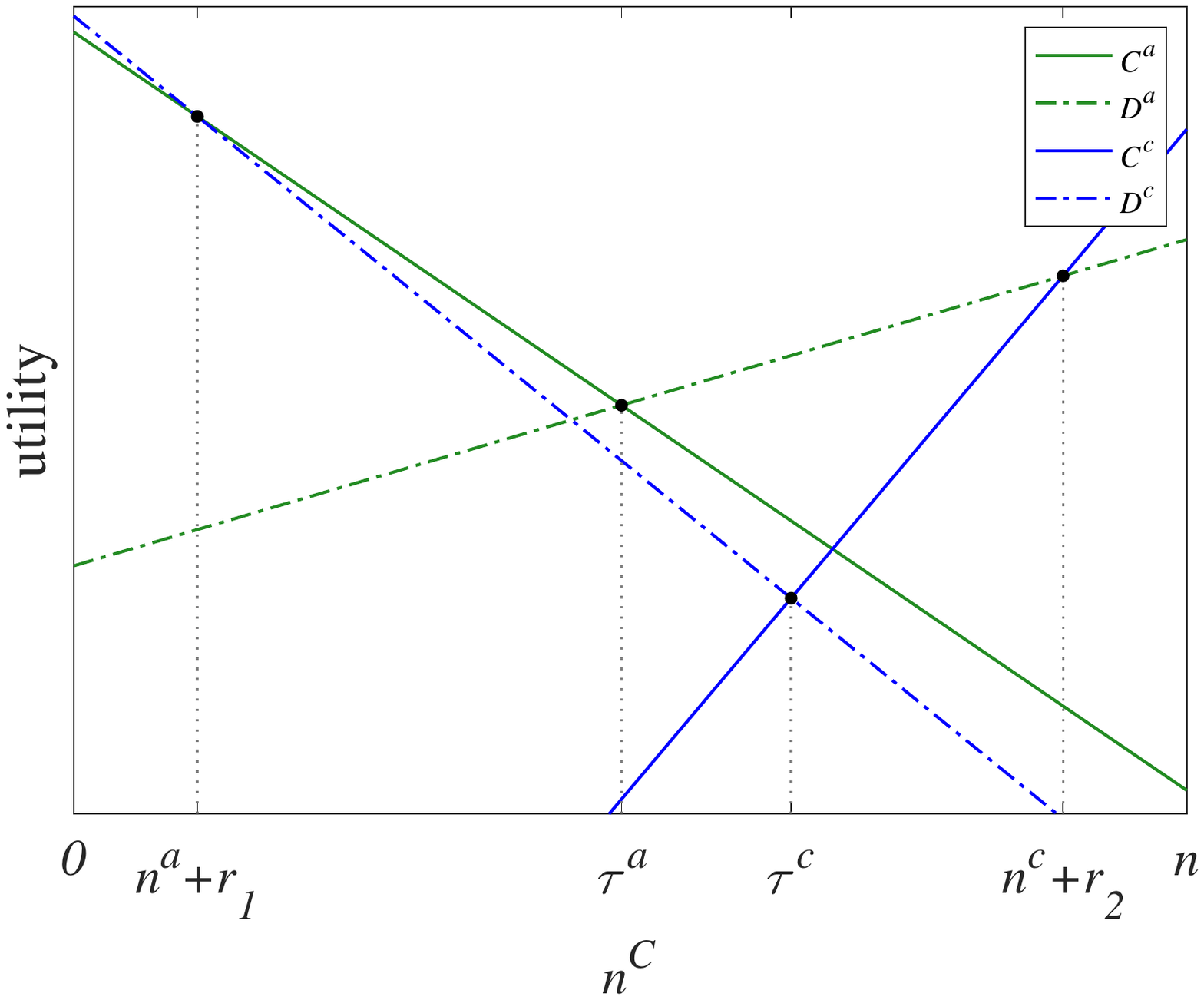}
            \caption{
           {\boldsymbol{$0<\tau^a<\tau^c<n$}}
            }
            \label{fig:schematic_utility_lines_left} 
        \end{subfigure}
        \hfill
        \begin{subfigure}[t]{0.49\textwidth}
            \centering
            \includegraphics[trim ={2.5cm 6cm 2cm 6cm}, clip, width=\textwidth]{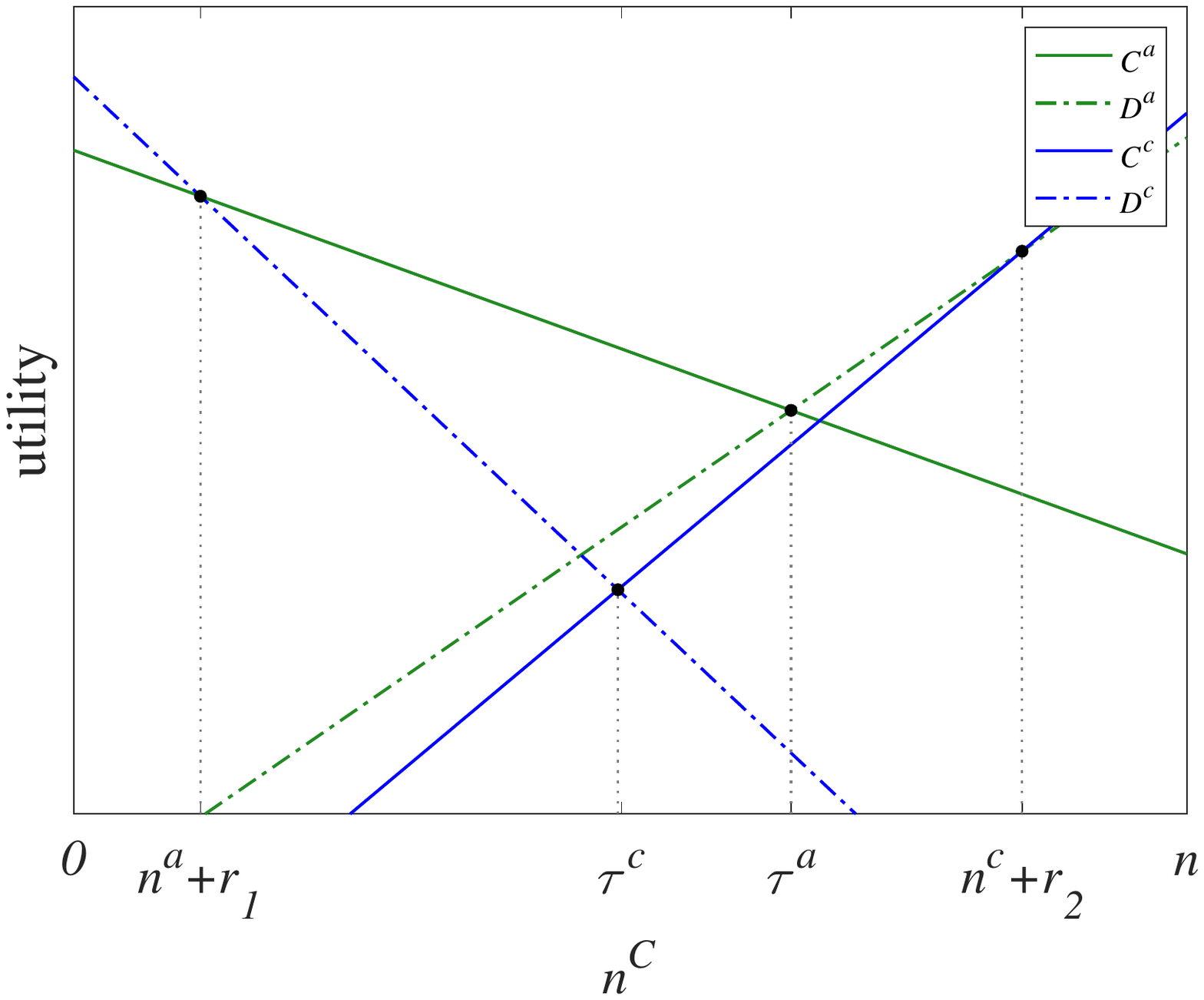}
            \caption{
            {$\boldsymbol{0<\tau^c<\tau^a<n}$}
              }
               \label{fig:schematic_utility_lines_right} 
               \end{subfigure}
        \caption{Utility lines of population dynamics with two mixed equilibria that their extreme states are equilibrium points.}
    \label{fig:schematic_utility_lines}
    \end{figure*}
        
\begin{theorem}\label{thm:stochastically_stable_equilibria}
Consider the unperturbed dynamics of a mixed binary-type population and the described regular perturbation. Suppose that in the unperturbed dynamics, the corresponding extreme state of each mixed equilibrium is an equilibrium. Then the set of stochastically stable equilibria is either empty or includes an extreme equilibrium.
\end{theorem}

\begin{proof}
The result is trivial should the dynamics be empty of equilibrium. Assume that there exists a mixed equilibrium $\x$. According to the assumptions, its corresponding extreme state $\y$ is an extreme equilibrium. Since $\{\y\}$ is an invariant set (a recurrent class) of the unperturbed dynamics, the Markov chain may not leave it unless an agent makes a mistake, resulting in $R(\y)\geq 1$. We prove that $c(\x,\y)=1$, and thus,     
\begin{align*}
    R(\y)\geq c(\x,\y).
\end{align*}
    Then by means of~\cref{lem:c_instead_of_c*}, we conclude that either $\x$ is not stochastically stable, or both $\x$ and $\y$ are. 
    
    Note that since $\x$ is an equilibrium, $c(\x,\y)\geq 1$. So to show $c(\x,\y)=1$, it suffices to introduce a path of cost $1$ from $\x$ to $\y$.
        
    \emph{Case 1:} $\x=(x_1^I,n^a,x_2^I,0)$, and hence, $\y = (0,n^a,0,0)$. Since $r = x_1^I+x_2^I\geq1$, without loss of generality assume that $x_1^I\geq 1$. A mistake by an active anticoordinating imitator who is cooperating, leads to the state $(x_1^I-1,n^a,x_2^I,0)$. Since $C^a(r+n^a) = D^c(r+n^a)$, and $\y$ is an equilibrium, $D^c$ is the dominant utility line at all states in the form of $(z_1^I, n^a, z_2^I, 0)$ with $0\leq z_1^I+z_2^I\leq r-1$ (\cref{fig:schematic_utility_lines}). Therefore, there is a costless path from $(x_1^I-1,n^a,x_2^I,0)$ to $\y$, which can be obtained by activating cooperating imitators successively. So the total cost of the given path from $\x$ to $\y$ is one.
    
    \emph{Case 2:} $\x=(x_2^I,0,x_2^I,n^c)$, and hence, $\y = (m^a, 0, m^c, n^c)$.  Since $r = x_1^I+x_2^I\leq m-1$,  we assume without loss of generality that $x_I^I\leq m^a-1$. A mistake by an active anticoordinating imitator which is defecting, changes the state of the system to $(x_1^I+1,0,x_2^I,n^c)$. $C^c(r+n^c) = D^a(r+n^c)$, and at $\y$ the dominant utility function is $C^c$. Therefore, at all states in form of $(z_1^I, 0, z_2^I, n^c)$, with $r+1\leq z_1^I+z_2^I\leq m$, $C^c$ is the dominant utility line (\cref{fig:schematic_utility_lines}). Hence, there exists a costless path from $(x_1^I+1, 0, x_2^I, n^c)$ to $\y$, which can be obtained by activating defecting imitators one after another. So we have a path of cost $1$ from $\x$ to $\y$.
\end{proof}

\begin{corollary}\label{cor:stochastically_stable_equilibria}
Suppose that the mixed binary-type population dynamics admit the extreme equilibria $(0, n^a, 0, 0)$ and $(m^a, 0, m^c, n^c)$. Then, the maximal stochastically stable set of the corresponding regular perturbation is a non-empty union of the non-singleton minimal invariant sets of the unperturbed dynamics or includes an extreme equilibrium.
\end{corollary}
\begin{proof}
By assumption, the corresponding extreme state of each mixed equilibrium of the population dynamics is an equilibrium. If the maximal stochastically stable set is empty of equilibria, it is a non-empty union of the non-singleton minimal invariant sets of the unperturbed dynamics. Otherwise, by~\cref{thm:stochastically_stable_equilibria}, at least one of the extreme equilibria is stochastically stable.
\end{proof}
\begin{remark}
Some facts justify our choices in~\cref{def:corresponding_extreme} for assigning an extreme state to each mixed equilibrium. For instance, $(0, 0, 0, 0)$ and $(m^a, n^a, m^c, n^c)$ are not ``in the neighborhood'' of any mixed equilibrium. Moreover, assuming that $(x_1^I, n^a, x_2^I, 0)$ is a mixed equilibrium, if $(m^a, n^a, m^c, 0)$ is an extreme equilibrium, then so is $(0, n^a, 0, 0)$ (\cref{fig:schematic_utility_lines}). Hence, assigning $(m^a, n^a, m^c, 0)$ to $(x_1^I, n^a, x_2^I, 0)$ as its corresponding extreme state would yield a weaker result compared to~\cref{thm:stochastically_stable_equilibria}.
\end{remark}

\subsubsection{Examples}\label{subsec:SS_examples}
The following three examples illustrate some of the possible configurations of the maximal stochastically stable set for mixed binary-type population dynamics satisfying the assumptions of~\cref{thm:stochastically_stable_equilibria}.

\begin{example}{\emph{[A single extreme equilibrium]}}\label{example:7_1}
    Assume that $(\tr{m^a},\tg{n^a},\tr{m^c},\tb{n^c})=(\tr{2},\tg{1},\tr{1},\tb{5})$, and the payoff matrices are set so that the utility functions are the followings:
  \begin{align*}
  &C^a(n^C) =-2n^C+\frac{51}{5},&&D^a(n^C)=5n^C-\frac{37}{2};\\
  &C^c(n^C) = \frac{33}{5}n^C-\frac{297}{10},&&D^c(n^C) = -\frac{14}{5}n^C+\frac{63}{5}. 
  \end{align*}
Hence, the tempers are $(\tg{\tau^a},\tb{\tau^c}) = (\tg{4.1}, \tb{4.5})$, and it can be shown that the system has the eight equilibria
\begin{align*}
&(\tr{0},\tg{1},\tr{0},\tb{0}), (\tr{1},\tg{1},\tr{1},\tb{0}),(\tr{2},\tg{1},\tr{0},\tb{0}),(\tr{2},\tg{1},\tr{1},\tb{0}),\\
&(\tr{0},\tg{0},\tr{0},\tb{5}), (\tr{1},\tg{0},\tr{1},\tb{5}),(\tr{2},\tg{0},\tr{0},\tb{5}),(\tr{2},\tg{0},\tr{1},\tb{5}),
\end{align*}    
and no non-singleton minimal invariant set. From~\cref{thm:stochastically_stable_equilibria} and the fact that the maximal stochastically stable set is nonempty, we know that at least one of $(\tr{0},\tg{1},\tr{0},\tb{0})$ or $(\tr{0},\tg{0},\tr{0},\tb{5})$ must be stochastically stable.  We claim that $(\tr{0},\tg{1},\tr{0},\tb{0})$ is the only stochastically stable state.
\begin{figure}[!h]
\centering
\includegraphics[trim ={2.3cm 3.7cm 2.3cm 3.3cm}, clip, width=.7\textwidth]{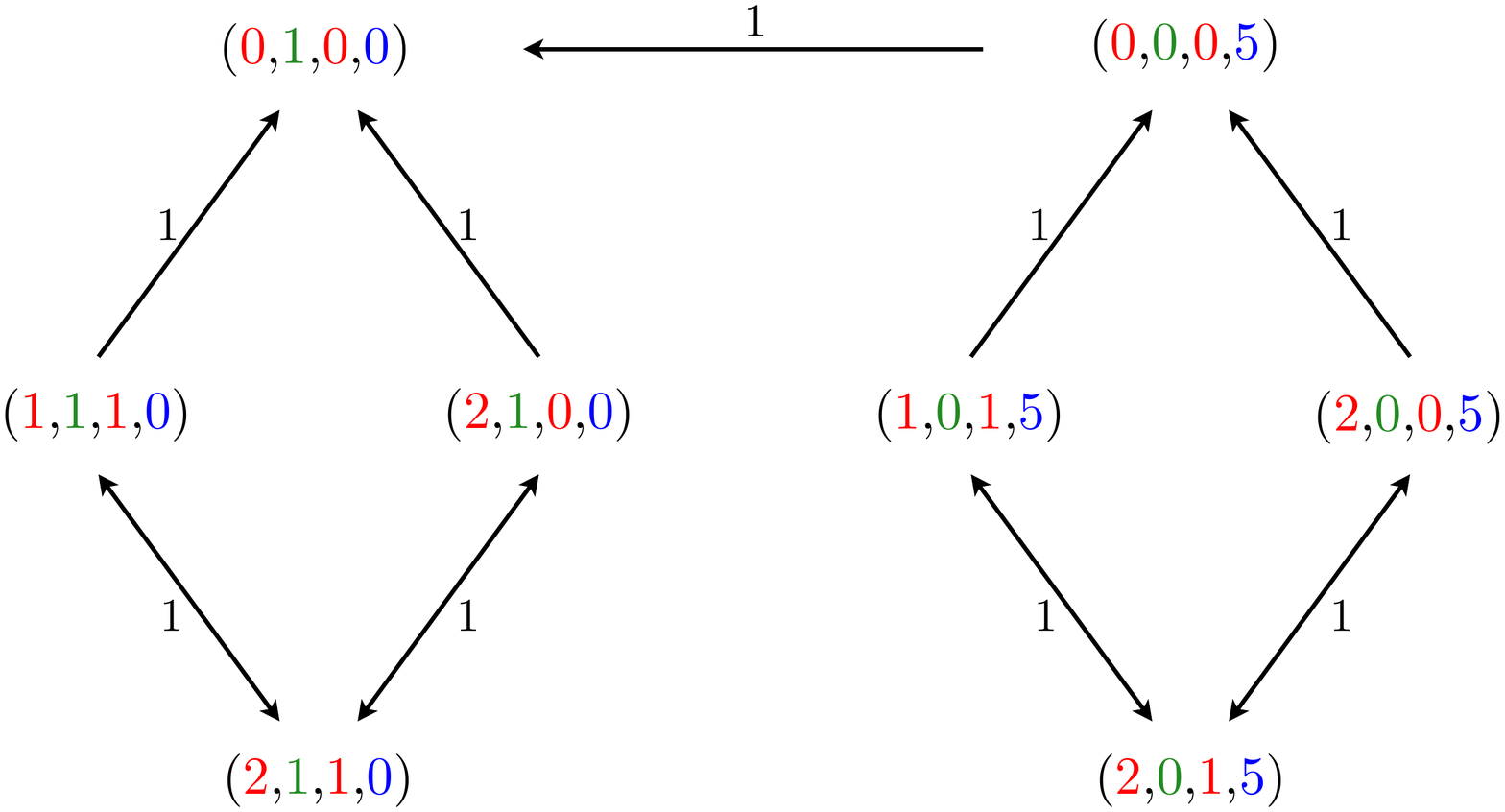}
\caption{Costs of some shortest paths between equilibrium points of the system in~\cref{example:7_1}.}
\label{fig:example_7_1}
\end{figure}

\cref{fig:example_7_1} shows a weighted digraph that is a subgraph of $\mathcal G(P^\varepsilon)$, introduced in~\cref{subsec:preliminaries}. Note that $\x\stackrel{c}{\longleftrightarrow}\y$ represents $\x\stackrel{c}{\longrightarrow}\y$ and $\y\stackrel{c}{\longrightarrow}\x$ simultaneously. To verify $(\tr{0},\tg{0},\tr{0},\tb{5})\stackrel{1}{\longrightarrow}(\tr{0},\tg{1},\tr{0},\tb{0})$, consider the following path:
\begin{align*}
(\tr{0},\tg{0},\tr{0},\tb{5})\stackrel{1}{\to}(\tr{0},\tg{0},\tr{0},\tb{4})\stackrel{0}{\to}(\tr{0},\tg{0},\tr{0},\tb{3})\stackrel{0}{\to}(\tr{0},\tg{0}&,\tr{0},\tb{2})\\
&\hspace*{-.13cm}\downarrow{\scriptstyle{0}}\\
(\tr{0},\tg{0}&,\tr{0},\tb{1})\stackrel{0}{\to}(\tr{0},\tg{0},\tr{0},\tb{0})\stackrel{0}{\to}(\tr{0},\tg{1},\tr{0},\tb{0}).\label{eq:example_7_1_3}
\end{align*}
Numbers on top of or next to arrows denote the cost of the corresponding transition. The first transition needs exactly one mistake by an active conformist. Therefore,
\begin{align*}
c((\tr{0},\tg{0},\tr{0},\tb{5}),(\tr{0},\tg{1},\tr{0},\tb{0}))=1.
\end{align*}
Other costs in~\cref{fig:example_7_1} are straightforward to calculate. 

We show that $R((\tr{0},\tg{1},\tr{0},\tb{0}))>1$. Indeed, setting $\x = (\tr{0},\tg{1},\tr{0},\tb{0})$, we show that for a state $\y$, if $c(\x,\y)= 1$, then $\y\in \mathcal D(\x)$. So starting from $\x$, more than one tremble is needed to leave $\mathcal D(\x)$. Suppose there is a path of cost $1$ from $\x$ to a state $\y$.  We assume, without loss of generality, that the second state on this path is not $\x$ itself. Hence, the first step needs exactly one mistake, while other steps must be costless. Therefore, the second state would be one of the following points:
\begin{equation*}
(\tr{1},\tg{1},\tr{0},\tb{0}), (\tr{0},\tg{0},\tr{0},\tb{0}), (\tr{0},\tg{1},\tr{1},\tb{0}), (\tr{0},\tg{1},\tr{0},\tb{1}).
\end{equation*}
It is straightforward to see that each of these states and consequently $\y$ belong to $\mathcal D(\x)$ (for instance, activation of any agent at $(\tr{1},\tg{1},\tr{0},\tb{0})$ leads to either the same state or $\x$ under the unperturbed dynamics). Now, by~\cref{lem:c_instead_of_c*}, $(\tr{0},\tg{0},\tr{0},\tb{5})$ is not stochastically stable, since 
\begin{align*}
R((\tr{0},\tg{1},\tr{0},\tb{0})) > c((\tr{0},\tg{0},\tr{0},\tb{5}), (\tr{0},\tg{1},\tr{0},\tb{0})).
\end{align*}
Moreover, the weighted edge $(\tr{0},\tg{0},\tr{0},\tb{5})\stackrel{1}{\longrightarrow}(\tr{0},\tg{1},\tr{0},\tb{0})$ implies $R((\tr{0},\tg{0},\tr{0},\tb{5})) \leq1$. But it is impossible to leave an equilibrium without bearing a positive cost. Hence, $R((\tr{0},\tg{0},\tr{0},\tb{5})) = 1$. So, by~\cref{lem:c_instead_of_c*}, stochastic stability of $(\tr{1},\tg{0},\tr{1},\tb{5})$ results in stochastic stability of $(\tr{0},\tg{0},\tr{0},\tb{5})$, which we rejected earlier. Thus $(\tr{1},\tg{0},\tr{1},\tb{5})$ is not stochastically stable. 

Similar arguments show that neither of $(\tr{2},\tg{0},\tr{0},\tb{5})$, $(\tr{2},\tg{0},\tr{1},\tb{5})$, $(\tr{2},\tg{1},\tr{0},\tb{0})$, $(\tr{1},\tg{1},\tr{1},\tb{0})$, nor $(\tr{2},\tg{1},\tr{1},\tb{0})$ is stochastically stable. Since the maximal stochastically stable set is non-empty, its only element is $(\tr{0},\tg{1},\tr{0},\tb{0})$.
\end{example}

\begin{example}{\emph{[No equilibria]}}\label{example:7_2}
  Suppose $(\tr{m^a},\tg{n^a},\tr{m^c},\tb{n^c})=(\tr{2},\tg{1},\tr{2},\tb{3})$, and the utility functions are as follows:
  \begin{align*}
  &C^a(n^C) =6,&&D^a(n^C) = 3n^C-\frac{81}{10};\\
  &C^c(n^C) = 20n^C-122,&&D^c(n^C) = -\frac{60}{41}n^C+\frac{366}{41}. 
  \end{align*}
  It follows that $(\tg{\tau^a},\tb{\tau^c}) = (\tg{4.7}, \tb{6.1})$. All equilibrium states of the dynamics together with its only non-singleton minimal invariant set, $\Omega$, are shown in~\cref{fig:example_7_2}. We show that $R(\Omega)>1$. By~\cref{lem:c_instead_of_c*}, this inequality in addition to the given costs in~\cref{fig:example_7_2}, implies that the maximal stochastically stable set is
 \begin{align*}
  \Omega = \left\{(\tr{2},\tg{0},\tr{2},\tb{0}),(\tr{2},\tg{1},\tr{2},\tb{0})\right\}.
 \end{align*}
  
\begin{figure}[!h]
\centering
\includegraphics[trim ={3.4cm 2.4cm 3.4cm 2.7cm}, clip, width=.7\textwidth]{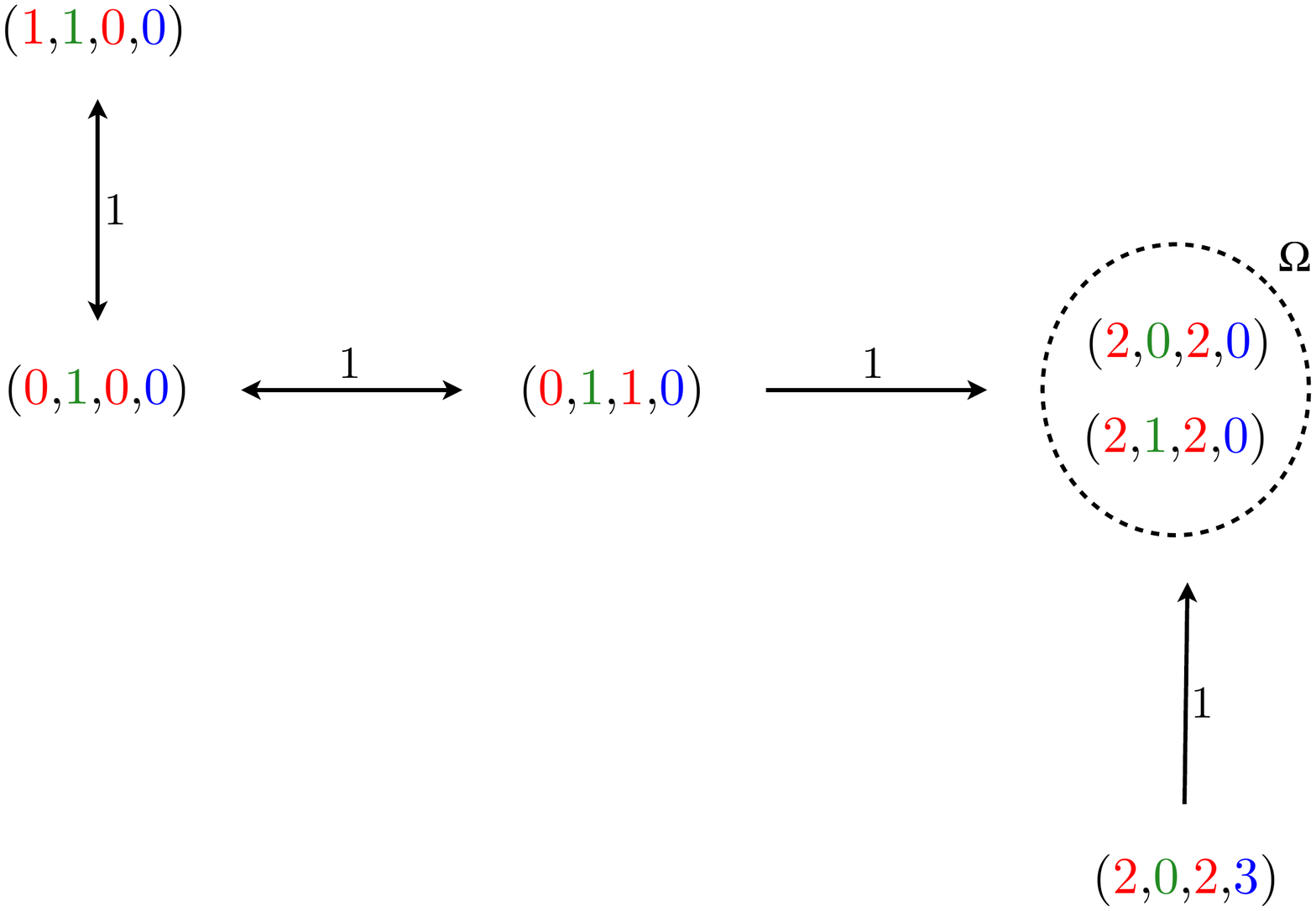}
\caption{Costs of some shortest paths between minimal invariant sets of the population dynamics of~\cref{example:7_2}.}
\label{fig:example_7_2}
\end{figure}
It is straightforward to show that 
\begin{align*}
\mathcal D(\Omega) =& \left\{(\tr{2},\tg{0},\tr{2},\tb{0}),(\tr{2},\tg{1},\tr{2},\tb{0}), (\tr{1},\tg{1},\tr{1},\tb{0}), (\tr{0},\tg{1},\tr{2},\tb{0}), (\tr{1},\tg{0},\tr{2},\tb{0}), (\tr{2},\tg{0},\tr{1},\tb{0}),\right.\\
&\ \ (\tr{2},\tg{1},\tr{0},\tb{0}), (\tr{1},\tg{1},\tr{2},\tb{0}), (\tr{2},\tg{1},\tr{1},\tb{0}), (\tr{1},\tg{1},\tr{1},\tb{1}), (\tr{0},\tg{1},\tr{2},\tb{1}),(\tr{1},\tg{0},\tr{2},\tb{1}),\\
&\ \ \!\!\left.(\tr{2},\tg{1},\tr{0},\tb{1}), (\tr{2},\tg{0},\tr{1},\tb{1}), (\tr{1},\tg{1},\tr{2},\tb{1}), (\tr{2},\tg{1},\tr{1},\tb{1}), (\tr{2},\tg{0},\tr{2},\tb{1}), (\tr{2},\tg{1},\tr{2},\tb{1})\right\},
\end{align*}
and that
\begin{align*}
\{\x: c(\Omega,\x) = 1\} = &\ \mathcal D(\Omega)\setminus\Omega\\
\subseteq &\ \mathcal D(\Omega).
\end{align*}
Hence, $R(\Omega) >1$.
\end{example}
\begin{example}{\emph{[Union of all minimal invariant sets]}}\label{example:7_3}
    Suppose $(\tr{m^a},\tg{n^a},\tr{m^c},\tb{n^c})=(\tr{1},\tg{2},\tr{2},\tb{4})$, and the utility functions are
 \begin{align*}
  &C^a(n^C) =-n^C+\frac{109}{10},&&D^a(n^C)=3n^C-\frac{87}{10};\\
  &C^c(n^C) = \frac{105}{31}n^C-\frac{693}{62},&&D^c(n^C) = -\frac{79}{3}n^C+\frac{869}{10}. 
  \end{align*}
As a result, we have $(\tg{\tau^a},\tb{\tau^c}) = (\tg{4.9}, \tb{3.3})$. All equilibria of the dynamics and its unique minimal invariant set of multiple elements, $\Omega$, are depicted in~\cref{fig:example_7_3}.
\begin{figure}[!h]
\centering
\includegraphics[trim ={3cm 4cm 3cm 4cm}, clip, width=.7\textwidth]{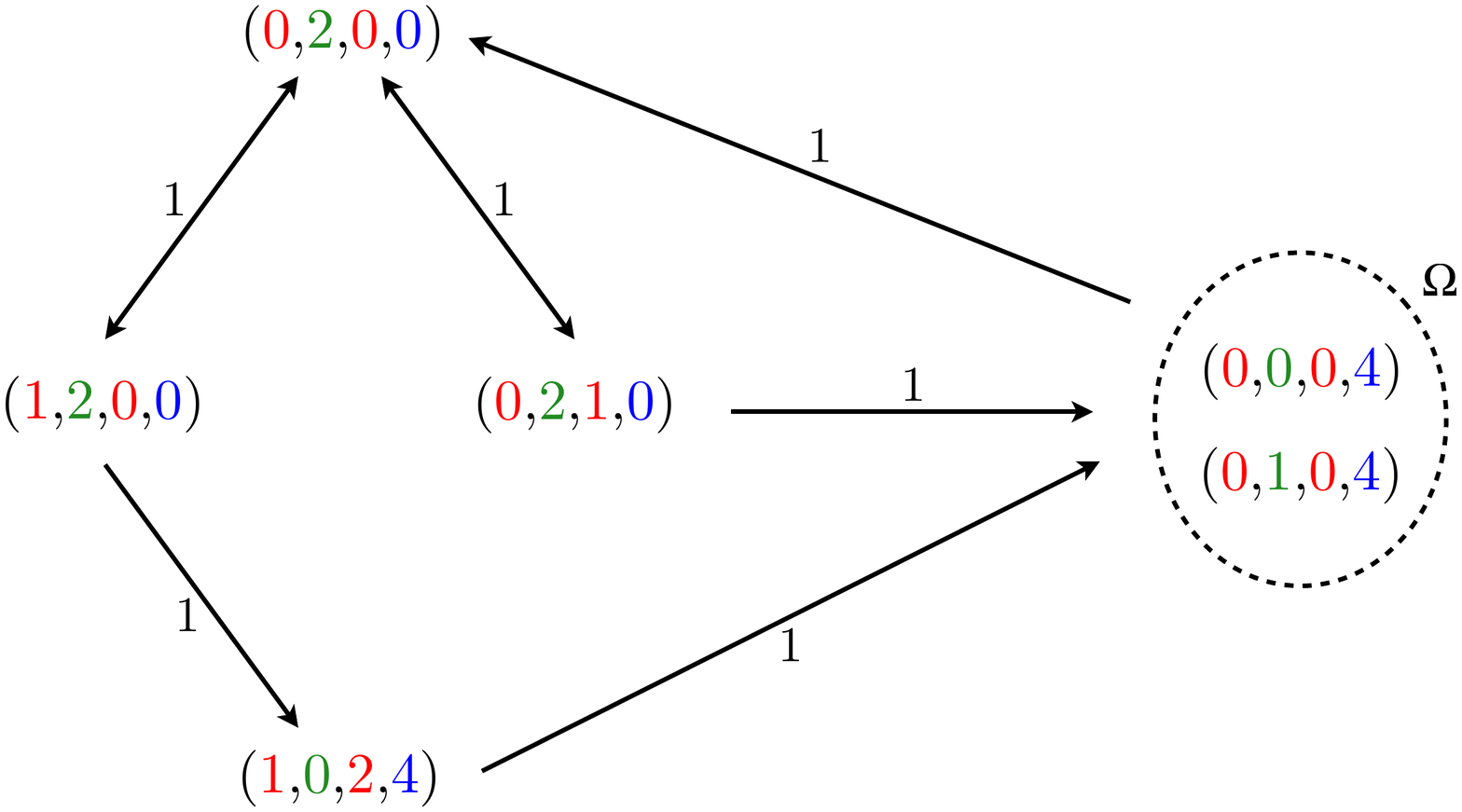}
\caption{Costs of some shortest paths between minimal invariant sets of the population dynamics of~\cref{example:7_3}.}
\label{fig:example_7_3}
\end{figure}
We prove that the maximal stochastically stable set is the union of all five minimal invariant sets. Indeed, assuming the costs in~\cref{fig:example_7_3} and by multiple usage of~\cref{lem:c_instead_of_c*}, we deduce that stochastic stability of each minimal invariant set results in the stochastic stability of the other four sets. For instance, if $(\tr{0},\tg{2},\tr{0},\tb{0})$ is stochastically stable, then so is each of the two mixed equilibria. By $(\tr{1},\tg{0},\tr{2},\tb{4})\stackrel{1}{\longrightarrow}\Omega$, we have $R((\tr{1},\tg{0},\tr{2},\tb{4})) = 1$. So because of the edge $(\tr{1},\tg{2},\tr{0},\tb{0})\stackrel{1}{\longrightarrow}(\tr{1},\tg{0},\tr{2},\tb{4})$, stochastic stability of $(\tr{1},\tg{2},\tr{0},\tb{0})$ yields stochastic stability of $(\tr{1},\tg{0},\tr{2},\tb{4})$, which in turn, by a similar argument yields stochastic stability of $\Omega$. 

Finally, we provide some minimum cost paths corresponding to the less trivial costs in~\cref{fig:example_7_3}:
\begin{align*}
c((\tr{0},\tg{2},\tr{1},\tb{0}),\Omega) = 1:\hspace*{4cm}&\\
(\tr{0},\tg{2},\tr{1},\tb{0})\stackrel{1}{\to}(\tr{0},\tg{2},\tr{1},\tb{1})\stackrel{0}{\to}(\tr{0},\tg{2},\tr{1},\tb{2})\stackrel{0}{\to}(\tr{0},\tg{2}&,\tr{1},\tb{3})\\
&\hspace*{-.13cm}\downarrow{\!\scriptstyle0}\\
(\tr{0},\tg{2}&,\tr{1},\tb{4})\stackrel{0}{\to}
(\tr{0},\tg{1},\tr{1},\tb{4})\stackrel{0}{\to}(\tr{0},\tg{1},\tr{0},\tb{4})\in\Omega.
\end{align*}

\begin{align*}
c((\tr{1},\tg{2},\tr{0},\tb{0}),(\tr{1},\tg{0},\tr{2},\tb{4})) = 1:\hspace*{3.2cm}&\\
(\tr{1},\tg{2},\tr{0},\tb{0})\stackrel{1}{\to}(\tr{1},\tg{2},\tr{0},\tb{1})\stackrel{0}{\to}(\tr{1},\tg{2},\tr{0},\tb{2})\stackrel{0}{\to}(\tr{1},\tg{2}&,\tr{0},\tb{3})\\
&\hspace*{-.13cm}\downarrow{\!\scriptstyle0}\\
(\tr{1},\tg{2}&,\tr{0},\tb{4})\\
&\hspace*{-.13cm}\downarrow{\!\scriptstyle0}\\
(\tr{1},\tg{2}&,\tr{1},\tb{4})\\
&\hspace*{-.13cm}\downarrow{\!\scriptstyle0}\\
(\tr{1},\tg{2}&,\tr{2},\tb{4})\stackrel{0}{\to}(\tr{1},\tg{1},\tr{2},\tb{4})\stackrel{0}{\to}(\tr{1},\tg{0},\tr{2},\tb{4}).
\end{align*}
\begin{flalign*}
\ \, c((\tr{1},\tg{0},\tr{2},\tb{4})&,\Omega) = 1:&&\\
&\hspace{.7cm}(\tr{1},\tg{0},\tr{2},\tb{4})\stackrel{1}{\to}(\tr{0},\tg{0},\tr{2},\tb{4})\stackrel{0}{\to}(\tr{0},\tg{0},\tr{1},\tb{4})\stackrel{0}{\to}(\tr{0},\tg{0},\tr{0},\tb{4})\in\Omega.&&
\end{flalign*}
\end{example}

The following example demonstrates that~\cref{thm:stochastically_stable_equilibria} cannot be generalized to all mixed binary-type populations.
\begin{example}\label{example:7_4}
  Consider a population with $(\tr{m^a},\tg{n^a},\tr{m^c},\tb{n^c})=(\tr{2},\tg{1},\tr{2},\tb{3})$. The payoff matrices are set to result in the following utility functions (\cref{fig:example_7_4}):
%
  \begin{figure}[!h]
    \centering
    \includegraphics[trim ={2.1cm 6cm 2cm 7.2cm}, clip, width=.7\textwidth, height = .35\textheight]{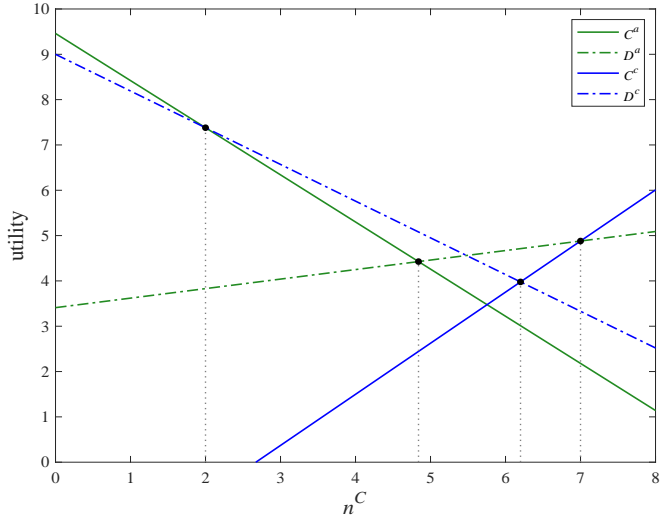}
    \caption{Utility functions of the conformists and nonconformists in~\cref{example:7_4}.}
    \label{fig:example_7_4}
\end{figure}
    \begin{align*}
        &C^a(n^C) = -\frac{26}{25} n^C+\frac{473}{50}, &&D^a(n^C) = \frac{21}{100} n^C+\frac{341}{100};\\
        &C^c(n^C) = \frac{451}{400} n^C-\frac{241}{80}, &&D^c(n^C) = -\frac{81}{100} n^C +9.
    \end{align*}
    Therefore, the tempers are $(\tg{\tau^a},\tb{\tau^c}) = (\tg{4.84}, \tb{6.2})$. The dynamics have three equilibria
    \begin{equation*}
\x = (\tr{1},\tg{1},\tr{0},\tb{0}),\quad \y=(\tr{0},\tg{1},\tr{1},\tb{0}),\quad \z=(\tr{2},\tg{0},\tr{2},\tb{3}).
\end{equation*}
    It is straightforward to see that there is no other minimal invariant set. Indeed, starting from each of the $72$ different states, we can find an activation sequence leading to one of the equilibria under the unperturbed dynamics. We estimate the minimum cost of transition from each of these equilibria to the others. 

\emph{Case 1:} Transition between two mixed equilibria. Consider the following path from $\x$ to $\y$:
\begin{align*}
\x\to(\tr{0},\tg{1},\tr{0},\tb{0})\to\y.
\end{align*}
The cost of the first step equals $1$ as it requires an anticoordinating imitator that is cooperating to become active and make a mistake. However, the second step is costless, for at $(\tr{0},\tg{1},\tr{0},\tb{0})$, cooperation is the optimal strategy for imitators (\cref{fig:example_7_4}). Hence, $c(\x, \y)=1$. Similarly, $c(\y,\x)=1$.

\emph{Case 2:} Transition between a mixed and an extreme equilibrium. Consider the following path from $\z$ to $\x$:
\begin{align*}
\z\stackrel{1}{\to}(\tr{1},\tg{0},\tr{2},\tb{3})\stackrel{0}{\to}(\tr{1},\tg{0},\tr{2},\tb{2})\stackrel{0}{\to}(\tr{1},\tg{0}&,\tr{2},\tb{1})\\
&\hspace*{-.13cm}\downarrow{\!\scriptstyle0}\\
(\tr{1},\tg{0}&,\tr{2},\tb{0})\stackrel{0}{\to}
(\tr{1},\tg{0},\tr{1},\tb{0})\stackrel{0}{\to}(\tr{1},\tg{1},\tr{1},\tb{0})\stackrel{0}{\to}\x.
\end{align*}
Since at $\z$, cooperation and defection are both optimal, the first step of this path takes place only if an anticoordinating imitator that is cooperating becomes active and trembles. All the later steps are costless (steps 5 and 7 are costless because of the optimality of defecting at their initial states). Hence, $c(\z,\x)=1$. Similarly, it can be shown that $c(\z,\y)=1$. We show that $c(\x,\z)\geq 2$. Suppose not. Then there is a path of cost $1$ from $\x$ to $\z$.  We assume, without loss of generality, that the second state on this path is not $\x$ itself. Hence, the first step needs exactly one mistake, while other steps are all costless. The second state would be one of the following points:
\begin{equation*}
(\tr{0},\tg{1},\tr{0},\tb{0}), (\tr{2},\tg{1},\tr{0},\tb{0}), (\tr{1},\tg{0},\tr{0},\tb{0}), (\tr{1},\tg{1},\tr{1},\tb{0}), (\tr{1},\tg{1},\tr{0},\tb{1}).
\end{equation*}
However, no costless path starting from one of these states leads to $\z$. To see this, again without loss of generality, we limit ourselves to paths without the same consecutive states. If the initial state is either $(\tr{0},\tg{1},\tr{0},\tb{0})$ or $(\tr{1},\tg{1},\tr{1},\tb{0})$, the second step would be one of $\x$ and $\y$. These are equilibrium states, and leaving them is impossible unless the active agent trembles. Starting from $ (\tr{2},\tg{1},\tr{0},\tb{0})$, the next state is $\x$. A costless path with the initial state $(\tr{1},\tg{0},\tr{0},\tb{0})$, has either $\x$, or $(\tr{1},\tg{0},\tr{1},\tb{0})$ as the second state. In the latter case, the third step is $(\tr{1},\tg{1},\tr{1},\tb{0})$, which was proved to lead to $\x$ or $\y$. Finally, if the path starts at $(\tr{1},\tg{1},\tr{0},\tb{1})$, in one step it goes to one of $\x$ and $(\tr{0},\tg{1},\tr{0},\tb{1})$. In the latter case, the next step will be $(\tr{0},\tg{1},\tr{0},\tb{0})$ that was investigated before. Similarly, it can be shown that $c(\y,\z)\geq 2$.

Case 1 together with Case 2 implies (see~\eqref{eq:gamma})
\begin{align*}
\gamma(\x)=\gamma(\y)=2\qquad\text{and} \qquad \gamma(\z)\geq 3.
\end{align*}
According to~\cref{prop:young_class_stability}, the maximal stochastically stable set is  $\{\x,\y\}$, which is empty of extreme equilibria.
\end{example}

\section{Concluding Remarks}    \label{sec_concludingRemarks}
We have studied mixed populations of imitators, conformists and nonconformists who asynchronously decide between cooperation and defection over time, resulting in the population dynamics. 
We have found the necessary and sufficient condition for a population state to be an equilibrium.
The equilibria take a similar structure to those of exclusive populations of conformists and exclusive populations of nonconformists.
That is, one can find an anticoordinating benchmark type $j_1$ and a coordinating benchmark type $j'_1$ such that at equilibrium, all nonconformists with thresholds higher than or equal to that of $j_1$ and conformists with thresholds lower than or equal to that of $j'_1$ cooperate and all other best-responders defect \cite{ramazi2018asynchronous, ramazi2020convergence}. The structure is also similar to those of exclusive populations of imitators, where either all imitators cooperate, all defect, or there are both a cooperating and defecting highest earner in the population. 

The stability results imply that in general populations, the extreme equilibria where the imitators either all cooperate or all defect are most likely the only equilibria that are stable against changes in strategies. Furthermore, the stochastic stability results imply that at least in the binary-type case, if the extreme state corresponding to each mixed equilibrium is an equilibrium, then at least one extreme equilibrium is stochastically stable; unless no other equilibrium is stochastically stable.
That is, under the stated condition, if the population dynamics are slightly noisy, where individuals, with a low frequency, play the opposite strategy than what their update rules imply, an extreme equilibrium will be visited infinitely often by the solution trajectory, and the frequency of this visitation does not vanish as the noise converges to zero.
These results highlight the robustness of extreme equilibria.

We have partially characterized the minimal invariant sets by finding necessary conditions. 
We have established the existence of anticoordinating benchmark types $j_1, j_2$ and coordinating benchmark types $j'_1,j'_2$ such that at the minimal invariant set, all nonconformists with thresholds higher than that of $j_1$ and all conformists with thresholds lower than that of $j'_1$ cooperate and all nonconformists with thresholds lower than that of $j_2$ and all conformists with thresholds higher than that of $j'_2$ defect.
If the minimum and maximum numbers of cooperators in the minimal  invariant set are known, then we know the benchmarks exactly, i.e., which agents will fix their strategies in the long run and which others will keep wandering. 
Particularly, the wandering conformists will infinitely often all simultaneously cooperate and infinitely often all simultaneously defect, when the solution trajectory enters the invariant set. 
The wandering nonconformists are less harmonious and seem to be the key in finding the extremum number of cooperators in the minimal  invariant set, 
which in turn may lead to an explicit formulation of the set and is left as future work.

 \section*{Acknowledgments}
  This work was funded in part by Alberta Environment and Parks.
 We would like to thank Prof. Russell Greiner for a generous funding and Prof. John Bowman for his support. 

\bibliographystyle{siamplain}
\bibliography{references}

\appendix
\section{Modified cost function and proof of~\cref{lem:c_instead_of_c*}}\label{app:modified_cost}
The proof of~\cref{lem:c_instead_of_c*}, as mentioned before, is based on~\cite[Theorem 3]{ellison2000basins}. Indeed, the two propositions are almost the same, with the distinction that in~\cite[Theorem 3]{ellison2000basins} the cost function, defined in~\cref{subsec:preliminaries}, is replaced with its modified version. This modification is done by Ellison as follows.

For an arbitrary state $\x$ and a recurrent class $\Omega$, consider the path $(\z_1\ldots,\z_T)\in\Pi(\x,\Omega)$, and suppose that $\Omega_{i_1}$,  $\Omega_{i_2}$, $\ldots$, $\Omega_{i_r}=\Omega$ is the sequence of recurrent classes through which the path passes consecutively (with the convention that a recurrent class can appear on the list multiple times but not successively). We define $c^*$ by 
 \begin{align}\label{eq:path_modified_cost}
 c^*(\z_1,\ldots,\z_T) = c(\z_1,\ldots,\z_T) - \sum_{j=2}^{r-1}R(\Omega_{i_j}),
 \end{align}
and set 
\begin{align*}
c^*(\x,\Omega) = \min_{(\z_1,\ldots,\z_T)\in\Pi(\x,\Omega)}c(\z_1,\ldots,\z_T).
\end{align*}
 
 For us, in spite of~\cite{ellison2000basins}, $c^*$ is not an essential object. So we will not try to justify its definition, and refer the reader to~\cite{ellison2000basins}. However, we provide an evolutionary intuition given in the same paper. As Ellison explains, the given modification of $c$ ``formalizes the observation that large evolutionary changes will occur more rapidly if it is possible for the change to be effected via a series of more gradual steps between nearly stable states''. 
 
Ellison proves the following result \cite[Theorem 3]{ellison2000basins}. 
 \begin{proposition}\label{prop:modified_cost}
 Let $P^\varepsilon$ be a regular perturbation of $P$.  For a recurrent class $\Omega$ of $P$, and a state $\x\not\in\Omega$, if we have $R(\Omega)>c^*(\x,\Omega)$, then $\bm{\mu}^*(\x)=0$. If $R(\Omega)=c^*(\x,\Omega)$, then $\bm{\mu}^*(\x)>0$ implies $\bm{\mu}^*(\Omega)>0$.
 \end{proposition}
 
\subsection{Proof of~\cref{lem:c_instead_of_c*}}
\begin{proof}
According to~\eqref{eq:c_dominance}, for a state $\x$ and a recurrent class $\Omega$, we have 
\begin{align}\label{eq:c_dominance}
c(\x,\Omega)\geq c^*(\x,\Omega).
\end{align}
Hence, assuming $\x\not\in\Omega$, if $R(\Omega)>c(\x,\Omega)$, it is concluded that $R(\Omega)>c^*(\x,\Omega)$, and~\cref{prop:modified_cost} yields $\mu^*(\x) = 0$. On the other hand, If $R(\Omega) = c(\x,\Omega)$ and $\mu^*(\x) >0$, then, by the first part of~\cref{prop:modified_cost}, in~\eqref{eq:c_dominance} the equality must hold. So $R(\Omega) = c^*(\x,\Omega)$, and according to the second part of~\cref{prop:modified_cost}, it holds that $\mu^*(\Omega) >0$.
\end{proof}

\end{document}